\documentclass[12 pt]{amsart}
\usepackage{amscd}
\usepackage{amsmath}
\usepackage{amsthm}
\usepackage{amsfonts}
\usepackage{amssymb}
\usepackage{amsxtra}
\usepackage{color}

\def\C{\mathbb{C}}
\def\Z{\mathbb{Z}}
\def\N{\mathbb{N}}

\def\R{\mathbb{R}}

 \newtheorem{thm}{Theorem}[section]
 \newtheorem{cor}[thm]{Corollary}
 \newtheorem{lem}[thm]{Lemma}

 \newtheorem{rem}[thm]{Remark}

\newcommand{\be}{\begin{equation}}
\newcommand{\ee}{\end{equation}}
\newcommand{\bea}{\begin{eqnarray}}

\newcommand{\eea}{\end{eqnarray}}
\newcommand{\Bea}{\begin{eqnarray*}}
\newcommand{\Eea}{\end{eqnarray*}}

\newcommand{\esssup}[0]{\operatornamewithlimits{ess\,sup}}

\newcounter{cnt1}
\newcounter{cnt2}
\newcounter{cnt3}
\newcommand{\blr}{\begin{list}{$($\roman{cnt1}$)$}
 {\usecounter{cnt1} \setlength{\topsep}{0pt}
 \setlength{\itemsep}{0pt}}}
\newcommand{\bla}{\begin{list}{$($\alph{cnt2}$)$}
 {\usecounter{cnt2} \setlength{\topsep}{0pt}
 \setlength{\itemsep}{0pt}}}
\newcommand{\bln}{\begin{list}{$($\arabic{cnt3}$)$}
 {\usecounter{cnt3} \setlength{\topsep}{0pt}
 \setlength{\itemsep}{0pt}}}
\newcommand{\el}{\end{list}}

\title[Fourier multipliers on $H^n$]{Fourier multipliers on the Heisenberg group revisited
}
\author{Sayan Bagchi}
\date{}
\begin{document}

\baselineskip=17pt

\address{Stat-Math Unit, Indian Statistical Institute, Kolkata, India.}
\curraddr{Department of Mathematics and Statistics\\
 Indian Institute of Science Education and Research Kolkata\\
Mohanpur 741246, Nadia, West Bengal, India}
\email{sayansamrat@gmail.com}




\begin{abstract}
In this paper we give explicit expressions of differential-difference operators that satisfy the hypothesis of the general Fourier multiplier theorem associated to the Heisenberg groups proved by Mauceri and De Michele, for one dimension, and Lin, for higher dimension. We also give a much shorter proof of the above mentioned theorem. Moreover, we obtain a sharp weighted estimate for Fourier multipliers on the Heisenberg groups.

\end{abstract}


\maketitle

\section [Introduction]
{Introduction and the main results}
Given a bounded function $m$ on $\R^n$, let us consider the operator $T_m$ defined as follows:
$$\mathcal{F}(T_m f)(\xi)= m(\xi) \mathcal{F}(f)(\xi). $$
Here $\mathcal{F}$ stands for the Fourier transform on $\R^n$. By Plancherel theorem, one can immediately see that $T_m$ is bounded on $L^2(\R^n)$. We say that $m$ (or equivalently $T_m$) is an $L^p$ multiplier if $T_m$ can be extended to $L^p(\R^n)$ as a bounded linear operator. A sufficient condition for $L^p$ multipliers was given by H\"{o}rmander, and is stated below (see for instance \cite[Corollary 8.11]{D}).
\begin{thm}[H\"{o}rmander]
Let $k=[\frac{n}{2}]+1$. If $m\in C^k$ away from the origin and satisfies
$$\sup_R R^{|\beta|-\frac{n}{2}}\left(\int_{R<|\xi|< 2R}|D^\beta m(\xi)|^2d\xi\right)^{\frac{1}{2}}\leq C$$
for all $|\beta|\leq k$, then $T_m$ is an $L^p$ multiplier.
\end{thm}

Now let us consider the Heisenberg group $H^n$. One can also define a Fourier multiplier corresponding to the group Fourier transform on the Heisenberg group.
Let $M=\{M(\lambda)\in B(L^2(\R^n)):\lambda\in \R^*\}$ be a family of uniformly bounded operators. Then the operator  $T_M$ is defined as follows
$$\widehat{(T_Mf)}(\lambda)= M(\lambda)\hat{f}(\lambda).$$
Here $\hat{f}$ stands for the group Fourier transform on the Heisenberg group.
Again by the Plancherel formula for the group Fourier transform it is immediate that $T_M$ is bounded on $L^2(H^n)$. We are interested in knowing when $T_M$ can be extended as a bounded linear operator on $L^p(H^n)$. In \cite{MM}, Mauceri and De Michele first gave a sufficient condition for $L^p$ multipliers on the Heisenberg group for $n=1$.
Later Chin Cheng Lin (see \cite{L}) generalized their result for other dimensions. Also see \cite{BT} for some other interesting properties of Fourier multiplier on the Heisenberg group.

In \cite{L}, Lin decomposed $M(\lambda)$ in terms of certain partial isometries which form a basis of the space of all Hilbert-Schmidt operators acting on the Fock spaces. Then he expressed the ``difference-differential'' operators in terms of those decompositions. But such difference-differential operators looked very complicated. Also, the proof of Lemma 2 given in \cite{L}, which is the key point in his work, involved very long and technical calculations, even though he gave a proof of that lemma only for  some particular type of polynomials.

The goals of this paper are the following: firstly we will find explicit expressions for the difference-differential operators. Secondly we will give a much simpler proof of the multiplier theorem on the Heisenberg groups and will also reduce the number of difference-differential operators. Thirdly, we will prove a quantitative weighted estimate. In order to state our results first we have to set-up some notation.

Let us consider the annihilation and creation operators
$$A_j(\lambda)=\frac{\partial}{\partial \xi_j}+|\lambda|\xi_j,\qquad A_j^*(\lambda)=-\frac{\partial}{\partial \xi_j}+|\lambda|\xi_j$$
where $j=1,2,..., n$. They are very common operators used in quantum mechanics. The non-commutative derivations of any operator $m$ are given by
$$\delta_j(\lambda) m = |\lambda|^{-\frac{1}{2}}[m, A_j(\lambda)],$$
$$\bar{\delta_j}(\lambda)m=|\lambda|^{-\frac{1}{2}}[A_j^*(\lambda), m]$$
for j=1, 2,..., n. For multiindices $\alpha, \beta\in \N$, define
$$\delta^\alpha(\lambda)=\delta_1^{\alpha_1}(\lambda)\delta_2^{\alpha_2}(\lambda)...\delta_n^{\alpha_n}(\lambda),\qquad \bar{\delta}^\beta(\lambda)=\bar{\delta}_1^{\beta_1}(\lambda)\bar{\delta}_2^{\beta_2}(\lambda)...\bar{\delta}_n^{\beta_n}(\lambda).$$
Given a family of operators $\{m(\lambda):\lambda\in \R\}$, we now consider a new operator $\Theta(\lambda)$ defined as follows
\begin{multline*}
\Theta(\lambda)m(\lambda)\\
=\frac{d}{d\lambda}m(\lambda)+ \frac{1}{2\lambda}[m(\lambda), \xi\cdot \nabla]
+\frac{1}{4\lambda\sqrt{|\lambda|}}\sum_{j=1}^n (\delta_j(\lambda)m(\lambda)A_j^*(\lambda)+ \bar{\delta_j}(\lambda)m(\lambda)A_j(\lambda)).
\end{multline*}

Though the expression of $\Theta(\lambda)$ may look complicated, it corresponds to the difference-differential operator related to t-variable defined in \cite{MM} and \cite{L}. In fact, if $g$ is a Schwartz class function on $H^n$, one can check that $\widehat{(it g)}(\lambda)=\Theta(\lambda)\hat{g}(\lambda)$. The operator $\Theta(\lambda)$ appears implicitly in some other works also. For example, the operator $\Lambda$ appearing in \cite[Proposition 2.4]{LT} is the same as $\Theta(\lambda)$.

An operator-valued function $M:\R\setminus \{0\}\rightarrow B(L^2(\R^n))$ is said to be in $E^k(\R\setminus\{0\})$ if $\delta^\alpha(\lambda) \bar{\delta}^\beta(\lambda) \Theta^s(\lambda)$ are in $B(L^2(\R^n))$ for all $|\alpha|+|\beta|+2s\leq k$ and $\lambda\in \R\setminus\{0\}$.
We also need the following dyadic projections
$$\chi_N(\lambda)=\sum_{2^N\leq(2k+n)|\lambda|<2^{N+1}} P_k(\lambda)$$
where $P_k(\lambda)$ is the projection on the eigenspace corresponding to the eigenvalue $(2k+n)|\lambda|$ of the scaled Hermite operator $H(\lambda)=-\Delta+\lambda^2|x|^2.$

We will prove the following result.
\begin{thm}\label{th:Main}Let $M$ be an operator-valued function which belongs to $E^k(\R\setminus \{0\})$, $k\geq 2[\frac{n+3}{2}]$. Also, assume
$$\sup_{\lambda\in \R \setminus \{0\}}\|M(\lambda)\|\leq C.$$ If $M$ satisfies
$$\sup_{N>0}2^{N(l-n-1)}\int^\infty_{-\infty}\||\lambda|^{-\frac{|\alpha|+|\beta|}{2}}\delta^\alpha(\lambda) \bar{\delta}^\beta(\lambda)\Theta^s(\lambda)M(\lambda)\chi_N(\lambda)\|_{\operatorname{HS}}^2|\lambda|^n d\lambda \leq C$$
for all $\alpha, \beta\in \N^n$, $s\in \N$ satisfying $|\alpha|+|\beta|+2s=l\leq 2[\frac{n+3}{2}]$, then $T_M$ is weak type $(1,1)$ and bounded for $1<p<2$.

\end{thm}
The $A_2$ conjecture (now a theorem) was one of the well-known conjectures in harmonic analysis until T. Hyt\"{o}nen (\cite{H1}) solved this for all standard Calder\'on--Zygmund operators, showing the sharp quantitative $L^2(w)$ bound with a linear dependence in the $A_2$ constant $[w]_{A_2}$. For historical developments in this direction, see \cite{B,PV,P,OB,LPR,CMP,AL1,AV}. Recently, M. T. Lacey \cite{L1} extended Hyt\"onen's result to Dini-continuous operators and A. Lerner \cite{AL1} found a simple proof of Lacey's result. In this article we prove similar results for multipliers on the Heisenberg group.

A weight $w$ is a nonnegative locally integrable function defined on $H^n$.  Given $1<p<\infty$, the Muckhenhoupt class of weights $A_p$ consists of all $w$ satisfying
$$
[w]_{A_p}:=\sup_Q \langle w\rangle_{Q}\langle \sigma \rangle_{Q}^{p-1}<\infty,\quad \sigma:= w^{1-p'}
$$
where the supremum is taken over all cubes $Q$ in $H^n$. Here, $\langle w\rangle_{Q}=\frac{1}{|Q|}\int_Q w$. We will prove the following result.
\begin{thm}\label{th:Main1}Let $M$ be an operator-valued function with each entry in $E^k(\R\setminus \{0\})$, $k\geq 2[\frac{n+3}{2}]$. Also, assume
$$\sup_{\lambda\in \R \setminus \{0\}}\|M(\lambda)\|\leq C.$$ If $M$ satisfies
$$\sup_{N>0}2^{N(l-n-1)}\int^\infty_{-\infty}\||\lambda|^{-\frac{|\alpha|+|\beta|}{2}}\delta^\alpha(\lambda) \bar{\delta}^\beta(\lambda)\Theta^s(\lambda)M(\lambda)\chi_N(\lambda)\|_{HS}^2|\lambda|^n d\lambda \leq C$$
for all $\alpha, \beta\in \N^n$, $s\in \N$ satisfying $|\alpha|+|\beta|+2s=l\leq 2[\frac{n+3}{2}]$, then
$$\|T_Mf\|_{L^p(w)}\leq C [w]^{\max\{1,\frac{1}{p-2}\}}\|f\|_{L^p(w)}$$
for all $w\in A_{\frac{p}{2}}(H^n)$, $2<p<\infty$.
\end{thm}

We notice that if $n$ is even then the number of derivatives required in the theorems stated above equals to $n+2$. Also, when $n$ is odd, we have to consider one more extra derivative in the hypothesis because of some technical reasons. Therefore one could expect that the results are not sharp. Infact, for spectral multipliers associated to the sub-Laplacian, one can prove Theorem \ref{th:Main1} using $n+2$ derivatives (see \cite{BFP}) for any given $n$. Proving the above theorems for $n+1$ derivatives are interesting open problems.

The paper is  structured as follows. In Section \ref{sec:preli} we will discuss some preliminaries about Heisenberg groups. In Section \ref{sec:relat} we will show that the differential-difference operators defined here are actually similar to that of \cite{L}. Section \ref{sec:main} is devoted to proving some crucial estimates and also proving Theorem \ref{th:Main}. In Section \ref{sec:proofm}, Theorem \ref{th:Main1} will be proved.

\section{preliminaries}
\label{sec:preli}
Let us consider the Heisenberg group $H^n=\C^n\times \R$ equipped with the group operation
$$(z,t)(w,s)=\Big(z+w, t+s+\frac{1}{2}\Im z\cdot\bar{w}\Big).$$
$H^n$ is a two-step nilpotent Lie group whose center is $\{(0,t): t\in \R\}$. The Haar measure on $H^n$ is simply the Lebesgue measure $dzdt$ on $\C^n\times \R$. The homogeneous norm $|(z,t)|$ is given by $\left(\frac{1}{16}(\sum_{i=1}^n |z_i|^2)^2+t^2\right)^{\frac{1}{4}}$. We will use the notation: $\rho(z,t)=|(z,t)|^4$.

The representation theory of $H^n$ is well studied due to the Stone-von Neumann theorem. The representations which are trivial at the center are merely one dimensional. On the other hand, the representations which are non-trivial at the center are called Schr\"{o}dinger representations and for each $\lambda\in \R\setminus \{0\}$ they are explicitly given by
$$\pi_\lambda (z,t) \phi(\xi)= e^{i\lambda t}e^{i\lambda(x\cdot\xi+\frac{1}{2}x\cdot y)}\phi(\xi+y)$$
where $\phi\in L^2(\R^n)$, the corresponding Hilbert space.
The group Fourier transform of a function $f\in L^1(H^n)$ is given by
$$\hat{f}(\lambda)=\int_{H^n}f(z,t) \pi_{\lambda}(z,t)dzdt.$$
If $f^\lambda$ stands for the inverse Fourier transform of $f$ in the last variable, then the group Fourier transform can be written as
$$\hat{f}(\lambda)=\int_{\C^n} f^\lambda(z) \pi_\lambda(z)dz,$$
where $\pi_\lambda(z)=\pi_\lambda(z,0)$. This leads us to define the Weyl transform of a function $f$ on $L^1(\C^n)$ in the following way:
$$W_\lambda (f)=\int_{\C^n}f(z)\pi_\lambda(z)dz.$$
Thus we have the following relation between the group Fourier transform on the Heisenberg group and the Weyl Transform
$$\hat{f}(\lambda)=W_\lambda(f^\lambda).$$

For a given $g\in L^1\cap L^2(\C^n)$, it can be shown that $W_\lambda(g)$ is a Hilbert--Schmidt operator satisfying
$$\|g\|^2_{L^2}=(2\pi)^{-n}|\lambda|^n \|W_\lambda(g)\|^2_{\operatorname{HS}}.$$ In fact the map $g\rightarrow W_\lambda (g)$ can be extended as an isometric isomorphism from $L^2(\C^n)$ to $S_2$, the space of all Hilbert--Schmidt operators on $L^2(\R^n)$.

From the relation between  Fourier transform on the Heisenberg group and the Weyl transform it is clear that for any given $f\in L^1\cap L^2(H^n)$ and for any $\lambda\in \R^*$, $\hat{f}(\lambda)$ is also a Hilbert--Schmidt operator. The map $f\rightarrow \hat{f}(\lambda)$ extends as an isometric isomorphism from $L^2(H^n)$ to $L^2(\R^*, S_2, (2\pi)^{-n-1}|\lambda|^n d\lambda)$ and the Plancherel theorem can be read as
$$\|f\|^2_{L^2(H^n)}=(2\pi)^{-n-1}\int^\infty_{-\infty} \|\hat{f}(\lambda)\|_{HS}^2 |\lambda|^n d\lambda.$$

Now let $m\in B(L^2(\R^n))$. Consider the operator $T^\lambda_m$ defined as
$$
W_\lambda(T^\lambda_m f)= m W_\lambda (f).
$$
From the Plancherel formula, it is clear that $T^\lambda_m$ is bounded on $L^2(\C^n)$. If $T^\lambda_m$ can be extended as a bounded linear operator on $L^p(\C^n)$ them $m$ is called a (left) Weyl multiplier for $L^p(\C^n)$. Weyl multipliers are studied in \cite{M} and \cite{BT}. On the other hand, see \cite{T1} for more details on the Heisenberg group.

We already defined the Fourier multipliers on the Heisenberg group associated to a uniformly bounded family of operators in the Introduction. If $M=\{M(\lambda)\in B(L^2(\R^n)):\lambda\in \R^*\}$ is a family of operators which are uniformly bounded, it is shown in \cite{BT} that
$$T_M f(z,t)=\frac{1}{2\pi}\int_{-\infty}^{\infty} e^{-i\lambda t} T^\lambda_{M(\lambda)}f^\lambda(z) d\lambda$$
is bounded on $L^2(H^n)$, where $T^\lambda_{M(\lambda)}$ is the Weyl multiplier associated to the operators $M(\lambda)$.

The left invariant vector fields on the Heisenberg group are given by
$$ T= \frac{\partial}{\partial t},\quad X_j=\frac{\partial}{\partial x_j}+\frac{1}{2} y_j \frac{\partial}{\partial t},\quad Y_j=\frac{\partial}{\partial y_j}-\frac{1}{2}x_j\frac{\partial}{\partial t}.$$
The above vectors fields give rise to families of unbounded operators $Z_j(\lambda)$ and $\bar{Z}_j(\lambda)$ defined as
$e^{-i\lambda t} Z_j(\lambda)f(z)=\frac{1}{2}(X_j-i Y_j)(e^{-i\lambda t} f(z))$ and $e^{-i\lambda t} \bar{Z}_j(\lambda)f(z)=\frac{1}{2}(X_j+i Y_j)(e^{-i\lambda t} f(z))$ respectively. Thus we have the following explicit form for $Z_j(\lambda)$ and $\bar{Z_j}(\lambda)$
$$Z_j(\lambda)=\frac{\partial}{\partial z_j}+\frac{\lambda}{4}\bar{z}_j,\qquad\bar{Z}_j(\lambda)=\frac{\partial}{\partial \bar{z}_j}-\frac{\lambda}{4}z_j.$$
The following lemma is well-known.
\begin{lem}\label{lem:WZ1}
For any $\lambda>0$, $f\in L^2(\C^n)$ we have
\begin{enumerate}
\item $W_\lambda(Z_j(\lambda)f)= -\frac{i}{2}W_\lambda(f)A^*_j(\lambda)$, $
W_\lambda(\bar{Z_j}(\lambda)f)=-\frac{i}{2}W_\lambda(f)A_j(\lambda),$  \item $\lambda W_\lambda(z_jf)=i [W_\lambda (f),
A_j(\lambda)]$, $\lambda W_\lambda(\bar{z}_jf)=i
[A_j^*(\lambda),W_\lambda (f)].$
\end{enumerate}

\end{lem}
Note that for the case of $\lambda<0$, we have to replace $A_j(\lambda)$ by $-A_j^*(\lambda)$ and $A_j^*(\lambda)$ by $-A_j(\lambda)$ in the above lemma.

The next lemma gives a expression of derivatives of $T^\lambda_{m(\lambda)}$ which will be used later.
\begin{lem}\label{lem:WZ}$$
\frac{d^k}{d\lambda^k}T^\lambda_{m(\lambda)}f(z)=\sum_{|\alpha|+|\beta|+|\gamma|+|\rho|+2s=2k} C_{\alpha, \beta}|\lambda|^{-\frac{|\alpha|+|\beta|}{2}} T^\lambda_{\delta^\alpha(\lambda) \bar{\delta}^\beta(\lambda)\Theta^s(\lambda) m(\lambda)}(z^\gamma \bar{z}^\rho f)(z).$$
\end{lem}
\begin{proof}
We will assume $\lambda>0$ and prove this lemma only for $k= 1, 2$. For general $k$, it can be proved by induction.
We will show that
\be\label{eq:DT2}\frac{d}{d\lambda}T^\lambda_{m(\lambda)}=T^\lambda_{\Theta(\lambda)m(\lambda)}+\sum_{j=1}^n\Big(\frac{i}{4\sqrt{\lambda}} T^\lambda_{\delta_j (\lambda)m(\lambda)}(\bar{z_j}f^\lambda)-
\frac{i}{4\sqrt{\lambda}}T^\lambda_{{\bar{\delta_j}}(\lambda)m(\lambda)}(z_j f)\Big).
\ee
Now in \cite[Lemma 2.4]{BT} it was found that
\be\label{eq:DT} \frac{d}{d\lambda}T^\lambda_{m(\lambda)}=T^\lambda_{\frac{d}{d\lambda}m(\lambda)}+T^\lambda_{\frac{1}{2\lambda}[m(\lambda), \xi\cdot\nabla]}+\frac{1}{2\lambda}[B,T^\lambda_{m(\lambda)}]\ee
where $B=\sum_{j=1}^n(z_j \frac{\partial}{\partial z_j}+\bar{z}_j\frac{\partial}{\partial \bar{z}_j})$.
Now, as $z_j\frac{\partial}{\partial z_j}+\bar{z}_j\frac{\partial}{\partial \bar{z}_j}= z_j Z_j(\lambda)+\bar{z}_j\bar{Z}_j(\lambda)$, we have
$$[B,T^\lambda_{m(\lambda)}]=\sum_{j=1}^n [z_j Z_j(\lambda)+\bar{z}_j \bar{Z}_j(\lambda), T^\lambda_{m(\lambda)}].$$
Using Lemma \ref{lem:WZ1} one can easily see that
$$W_\lambda(z_j Z_j(\lambda) T^\lambda_{m(\lambda)}f)= \frac{1}{2\sqrt{\lambda}} \delta_j(\lambda)(m(\lambda)W_\lambda(f)A_j^*(\lambda)),$$
$$W_\lambda(\bar{z}_j \bar{Z}_j(\lambda)T^\lambda_{m(\lambda)}f=\frac{1}{2\sqrt{\lambda}} \bar{\delta}_j(\lambda)(m(\lambda)W_\lambda(f)A_j(\lambda)),$$
$$W_\lambda( T^\lambda_{m(\lambda)}(z_j Z_j(\lambda)f))=  \frac{1}{2\sqrt{\lambda}} m(\lambda)\delta_j(\lambda)\Big(W_\lambda(f)A_j^*(\lambda)\Big)$$
and
$$W_\lambda(T^\lambda_{m(\lambda)}(\bar{z}_j \bar{Z}_j(\lambda)f))=\frac{1}{2\sqrt{\lambda}}m(\lambda) \bar{\delta}_j(\lambda) \big(W_\lambda(f)A_j(\lambda)\Big).$$
Putting together the above relations we get
\begin{align*}W_\lambda([z_j \frac{\partial}{\partial z_j}+\bar{z}_j \frac{\partial}{\partial \bar{z}_j},T^\lambda_{m(\lambda)}]f)&=\frac{1}{2\sqrt{\lambda}} \Big( \bar{\delta}_j(\lambda) m(\lambda) \Big)W_\lambda(f) A_j(\lambda)+\frac{1}{2\sqrt{\lambda}} \Big(\delta_j(\lambda)m(\lambda)\Big)W_\lambda(f) A_j^*(\lambda)\\
&= \frac{1}{2} \Big(\bar{\delta}_j(\lambda)m(\lambda)\Big)\delta_j(\lambda)W_\lambda (f)+\frac{1}{2\sqrt{\lambda}}\Big(\bar{\delta_j}(\lambda)m(\lambda)\Big)A_j(\lambda)W_\lambda(f)\\
&\quad  -\frac{1}{2}\Big(\delta_j(\lambda)m(\lambda)\Big)\bar{\delta_j}(\lambda)W_\lambda (f)+\frac{1}{2\sqrt{\lambda}}\Big(\delta_j(\lambda)m(\lambda)\Big)A_j^*(\lambda)W_\lambda(f).
\end{align*}
Using \eqref{eq:DT} and the above equation we get our required result.

Now we will prove the lemma for $k=2$. From \eqref{eq:DT2} we have
$$\frac{d^2}{d\lambda^2}T^\lambda_{m(\lambda)}=\frac{d}{d\lambda}T^\lambda_{\Theta(\lambda)m(\lambda)}+\sum_{j=1}^n\frac{d}{d\lambda}\Big(\frac{i}{4\sqrt{\lambda}} T^\lambda_{\delta_j (\lambda)m(\lambda)}(\bar{z_j}f^\lambda)-
\frac{i}{4\sqrt{\lambda}}T^\lambda_{{\bar{\delta_j}}(\lambda)m(\lambda)}(z_j f)\Big).$$
Now the first term of the right hand side of the above equation can be dealt similarly as we have done previously for $k=1$. So, it enough to consider $\frac{d}{d\lambda}\Big(\frac{1}{\sqrt{\lambda}} T^\lambda_{\bar{\delta_j} (\lambda)m(\lambda)}(z_j f)\Big).$ Observe that
$$\frac{1}{\sqrt{\lambda}} T^\lambda_{\bar{\delta_j} (\lambda)m(\lambda)}(z_j f)(z)=-i\bar{z_j} T^\lambda_{m(\lambda)}(z_jf)(z)+ i T^\lambda_{m(\lambda)}(z_j\bar{z}_j f)(z).$$
Hence, $$\frac{d}{d\lambda}\Big(\frac{1}{\sqrt{\lambda}} T^\lambda_{\bar{\delta_j}(\lambda)m(\lambda)}(z_jf)\Big)(z)=-i\bar{z_j} \frac{d}{d\lambda}T^\lambda_{m(\lambda)}(z_jf)(z)+ i\frac{d}{d\lambda}T^\lambda_{m(\lambda)}(z_j\bar{z}_j f)(z). $$
Both the terms can be handled similarly as in the case $k=1$, and this will lead us to our desired result.
\end{proof}

Let $\Phi_\mu$, $\mu\in \N^n$, stand for the normalised Hermite functions and $$\Phi_\mu^\lambda(\xi)=|\lambda|^{\frac{n}{4}}\Phi_\mu(|\lambda|^\frac{1}{2}\xi).$$ Then it is well-known that
$$A_j(\lambda)\Phi_\mu^\lambda=\sqrt{2\mu_j|\lambda|}\Phi^\lambda_{\mu-e_j}\quad \text{ and } \quad A_j^*(\lambda)\Phi_\mu^\lambda=\sqrt{(2\mu_j+2)|\lambda|}\Phi^\lambda_{\mu+e_j}.$$

Now, we will consider the Hermite multipliers. For any bounded function $a$ on $\N\times\R^*$, they can be defined as
$$a(H(\lambda))=\sum_{k=0}^\infty a(k,\lambda)P_k(\lambda)$$
where $P_k(\lambda)$ is the projection on the eigenspace of $H(\lambda)$ corresponding to the eigenvalue $(2k+n)|\lambda|$. Let us also consider the finite difference operators acting on
$a$
$$\Delta_+a(k,\lambda)=a(k+1,\lambda)-a(k,\lambda)$$
$$\Delta_-a(k,\lambda)=a(k,\lambda)-a(k-1,\lambda).
$$

The next theorem is the $\lambda$-version of \cite[Lemma 2.1]{M} and will be used to prove the crucial estimate in this article.

\begin{lem}\label{lem:lambda version} Let $p, q\in \N^n$. Then there exists constants $C_{p,q,r}$ such that
$$\delta^p(\lambda)\bar{\delta}^q(\lambda) a(H(\lambda))=\sum C_{p,q,r} |\lambda|^{-\frac{|q|+2|r|-|p|}{2}} (A^*(\lambda))^{q+r-p}A^r(\lambda)(\Delta^{|r|}_- \Delta^{|q|}_+ a)(H(\lambda))$$
where the sum is extended over the set of multiindices $r\in \N^n$ such that $0\leq r\leq p\leq q+r$.
\end{lem}

\section{The relation between $\delta_j(\lambda)$, $\bar{\delta}(\lambda)$, $\Theta(\lambda)$ and the differential-difference operators defined in \cite{L} and \cite{MM} }
\label{sec:relat}

In this section we will show that the differential-difference operators that we have defined earlier are similar to the differential-difference operators defined in \cite{L}. The only difference between them is that in our case we have realised those operators on $L^2(\R^n)$ whereas C. C. Lin considered operators on Fock spaces, which are actually isomorphic to $L^2(\R^n)$.

In order to define the partial isometries let us first set up some notations. Let $(\lambda, m,\alpha)\in \R^*\times\Z^n\times \N^n$. Define
$$m_i^+=\max\{m_i,0\},\quad m_i^-= -\min\{m_i,0\},\N
$$
$$m^+=(m_1^+, m_2^+,\cdots, m_n^+),\quad m^-=(m_1^-, m_2^-,\cdots, m_n^-).$$
As defined in \cite{L} and \cite{MM}, the partial isometries on $L^2(\R^n)$ can be defined as follows
$$V^m_\alpha(\lambda)\Phi_\mu^\lambda=(-1)^{|m^+|}\delta_{\alpha+m^+,\mu} \Phi_{\alpha+m^{-}}^\lambda, \quad \mbox{when }\lambda>0$$
and
$$V^m_\alpha(\lambda)=[V^m_\alpha(-\lambda)]^*,\quad \mbox{when }\lambda<0.$$

Here, $\delta_{\alpha, \beta}$ stands for the Kronecker delta. Now let us consider a operator $M(\lambda)$ which is a finite linear combination of the partial isometries. That is,
$$M(\lambda)=\sum_{m,\alpha}B(\lambda, m, \alpha) V^m_\alpha(\lambda)$$
where the sum runs over a finite subset of $\Z^n\times \N^n$. We calculate $\delta_j(\lambda)M(\lambda)$, $\bar{\delta}_j(\lambda)M(\lambda)$ and for $\lambda>0$ ,$\Theta(\lambda)M(\lambda)$.
\begin{align*}
\delta_j(\lambda)M(\lambda)\Phi^\lambda_\mu&=|\lambda|^{-\frac{1}{2}}[M(\lambda),A_j(\lambda)]\Phi_\mu^\lambda\\
&= \sqrt{2}\sum_{m,\alpha}B(\lambda, m, \alpha)(\alpha_j+m^+_j+1)^{\frac{1}{2}}(-1)^{|m^+|}\delta_{\alpha+m^++e_j, \mu}\Phi^\lambda_{\alpha+m^-}\\&\quad-\sqrt{2}\sum_{m,\alpha}B(\lambda, m, \alpha)(\alpha_j+m^-_j)^{\frac{1}{2}}(-1)^{|m^+|}\delta_{\alpha+m^+, \mu}\Phi^\lambda_{\alpha+m^--e_j}.
\end{align*}
Now, if $m_j\geq 1$ for all $m$ appearing in the sum, then the above equals to
\be\label{eq:delta}\sqrt{2}\sum_{m,\alpha}B(\lambda, m-e_j, \alpha+e_j)(\alpha_j+1)^{\frac{1}{2}}V^m_\alpha(\lambda)\Phi^\lambda_{\mu}- \sqrt{2}\sum_{m,\alpha}B(\lambda, m-e_j, \alpha)(\alpha_j+m_j)^{\frac{1}{2}}V^m_\alpha(\lambda)\Phi^\lambda_{\mu},\ee
whereas if $m_j\leq 0$, then $\delta_j(\lambda)M(\lambda)\Phi^\lambda_\mu$ equals to
$$\sqrt{2}\sum_{m,\alpha}B(\lambda, m-e_j, \alpha-e_j)\alpha_j^{\frac{1}{2}}V^m_\alpha(\lambda)\Phi^\lambda_{\mu}-\sqrt{2}\sum_{m,\alpha}B(\lambda, m-e_j, \alpha)(\alpha_j-m_j+1)^{\frac{1}{2}}V^m_\alpha(\lambda)\Phi^\lambda_{\mu}.$$
From the above calculation we can easily see that $\frac{1}{\sqrt{|\lambda|}}\delta_j(\lambda)M(\lambda)$ is actually similar to $\Delta_{z_j}$ defined in \cite{L}. Similarly we can also show that $\frac{1}{\sqrt{|\lambda|}}\bar{\delta}_j(\lambda)M(\lambda)$ is actually similar to $\Delta_{\bar{z_j}}$ defined in \cite{L}.

We now look at $\Theta(\lambda)M(\lambda)\Phi_\mu$. First assume $m_j\geq 1$, for all $j$. In order to do so let us first calculate
$\frac{1}{4\lambda\sqrt{\lambda}}\delta_j(\lambda)M(\lambda)A^*(\lambda)\Phi^\lambda_\mu$, which equals to
$$\frac{1}{4\lambda}(2\mu_j+2)^{\frac{1}{2}}\delta_j(\lambda)M(\lambda)\Phi^\lambda_{\mu+e_j}.$$
 Using (\ref{eq:delta}) the above equals to
\begin{align*}&\frac{1}{2\lambda}\sum_{m,\alpha}(\alpha_j+m_j+1)B(\lambda,m,\alpha)V^m_\alpha(\lambda)\Phi^\lambda_{\mu}\\
&\quad -\frac{1}{2\lambda}\sum_{m,\alpha}\sqrt{(\alpha_j+1)(\alpha_j+m_j+1)}B(\lambda,m,\alpha+e_j)V^m_\alpha(\lambda)\Phi^\lambda_{\mu}.
\end{align*}
Similarly, we can show that
\begin{align*}
\frac{1}{4\lambda\sqrt{\lambda}}\bar{\delta_j}(\lambda)M(\lambda)A(\lambda)\Phi^\lambda_\mu&=\frac{1}{2\lambda}\sum_{m,\alpha}\sqrt{\alpha_j(\alpha_j+m_j)}B(\lambda,m,\alpha-e_j)V^m_\alpha(\lambda)\Phi^\lambda_{\mu}\\
&\quad -\frac{1}{2\lambda}\sum_{m,\alpha}(\alpha_j+m_j)B(\lambda,m,\alpha)V^m_\alpha(\lambda)\Phi^\lambda_{\mu}.
\end{align*}
Thus, if $m_j\geq 1$ for all $m\in \Z^n$ appearing in the sum of $M(\lambda)$, we have
\begin{align}
\label{eq:delta+*}
&\frac{1}{4\lambda\sqrt{\lambda}}\delta_j(\lambda)M(\lambda)A^*(\lambda)\Phi^\lambda_\mu+\frac{1}{4\lambda\sqrt{\lambda}}\bar{\delta_j}(\lambda)M(\lambda)A(\lambda)\Phi^\lambda_\mu\\
\notag &=\frac{1}{2\lambda}\sum_{m,\alpha}\sqrt{\alpha_j(\alpha_j+|m_j|)}B(\lambda,m,\alpha-e_j)V^m_\alpha(\lambda)\Phi^\lambda_{\mu}\\
\notag &\quad -\frac{1}{2\lambda}\sum_{m,\alpha}\sqrt{(\alpha_j+1)(\alpha_j+|m_j|+1)}B(\lambda,m,\alpha+e_j)V^m_\alpha(\lambda)\Phi^\lambda_{\mu}\\
\notag &\quad +\frac{1}{2\lambda} \sum_{m,\alpha}B(\lambda,m,\alpha)V^m_\alpha(\lambda)
\end{align}

Similarly we can check that if $m_j\leq 0$, for all $m\in \Z^n$ appearing in the sum of $M(\lambda)$, we will get the same result.

In order to calculate $\frac{d}{d\lambda}M(\lambda)$, let us observe that $$M(\lambda)=\delta_{\sqrt{\lambda}}\sum_{m,\alpha}B(\lambda, m, \alpha) V^m_\alpha(1)\delta_{\sqrt{\lambda}}^{-1},$$ where for any function $f$ on $\R^n$, $\delta_\lambda(f)(\xi)=f(\lambda \xi)$. Hence, it is easy to see that ( see \cite{JST})

\be\label{eq:dM}\frac{d}{d\lambda}M(\lambda)=\sum_{m,\alpha}\frac{\partial}{\partial \lambda}B(\lambda,m,\alpha)V^m_\alpha(\lambda)-\frac{1}{2\lambda}[M(\lambda),x\cdot\nabla].\ee
Hence from (\ref{eq:delta+*}) and (\ref{eq:dM}) and using the definition of $\Theta(\lambda)$ we have
\begin{align}
\label{eq:dM1}\notag\Theta(\lambda)M(\lambda)&=\sum_{m,\alpha}\frac{\partial}{\partial \lambda}B(\lambda,m,\alpha)V^m_\alpha(\lambda)+\frac{n}{2\lambda} \sum_{m,\alpha}B(\lambda,m,\alpha)V^m_\alpha(\lambda)\\
&\quad +\frac{1}{2\lambda}\sum_{m,\alpha}\sum_{j=1}^n\sqrt{\alpha_j(\alpha_j+|m_j|)}B(\lambda,m,\alpha-e_j)V^m_\alpha(\lambda)\\
&\notag\quad -\frac{1}{2\lambda}\sum_{m,\alpha}\sum_{j=1}^n\sqrt{(\alpha_j+1)(\alpha_j+|m_j|+1)}B(\lambda,m,\alpha+e_j)V^m_\alpha(\lambda)
\end{align}
which is similar to $\Delta_t$ defined in \cite{L}.

\section{Some kernel estimates and Proof of Theorem \ref{th:Main}}
\label{sec:main}
As it was done in \cite{MM} and \cite{L}, here also we will use \cite[Theorem 3.1]{CW} in order to prove Theorem \ref{th:Main}. For that we need to find a well-behaved approximate identity which satisfies certain estimate. Let us consider the function
$$\varphi^{n-1}_{k}(z)=L_k^{n-1}\Big(\frac{1}{2}|z|^2\Big)e^{-\frac{1}{4}|z|^2}$$
where $L_k^{n-1}$ are the usual Laguerre polynomials of type $n-1$ and degree $k$. Also, define $\varphi^{n-1}_{k,\lambda}(z)=\varphi^{n-1}_k(\sqrt{|\lambda|}z)$.   Let $\phi_r(z,t)$ be the Fourier transform on the $\lambda$-variable of the function $\phi^\lambda_r$ defined as follows
$$\phi^\lambda_r(z)=C_n\sum_k e^{-2r (2k+n)|\lambda|}|\lambda|^n\varphi^{n-1}_{k,\lambda}(z),$$
where $C^{-1}_n=\int_{H^n}\phi_1(z,t)dzdt$. $\phi_r$ will play the role of approximate identity. In fact, they satisfy the following properties.
\begin{lem}For each $r>0$, let $\phi_r$ be defined above. Then
\begin{enumerate}
\item $\phi_r(z,t)=r^{-n+1}\phi_1(r^{-\frac{1}{2}}z, r^{-1}t)$, $r>0$,
\item $\int_{H^n}|\phi_r(z, t)|(1+\frac{\rho(z,t)}{r^2})^\eta dzdt\leq~C$, for some $\eta>0$,\\
\item $\int_{H^n}\phi_r(z,t)dzdt=1,$\\
\item $\phi_r *\phi_s= \phi_s * \phi_r,$\\
\item $\int_{H^n}|\phi_r((z,t)(z_0,t_0)^{-1})-\phi_r(z,t)|dzdt\leq C~\big( \frac{\rho(z_0,t_0)}{r^2}\big)^\eta,$ where $(z_0, t_0)\in H^n,$
\item $ \phi_r(z,t)=\phi_r(-z,-t)$.
\end{enumerate}

\end{lem}
The proof of the lemma is same as \cite[Lemma 1]{L}. Once we have the above approximate identity, let us define $\psi_r=\phi_{\frac{r}{2}}-\phi_r$. Then  \cite[Theorem 3.1]{CW} tells that in order to prove Theorem \ref{th:Main} we only have to prove that there exists $\epsilon>0$ such that
\be\label{eq:Main}\int_{H^n}|T_M \psi_r(z,t)|\Big(1+\Big(\frac{\rho(z,t)}{r^2}\Big)^\epsilon\Big)dzdt<C.\ee

As shown in \cite{L}, we can also show that in order to prove (\ref{eq:Main}) it is enough to prove
\be\label{eq:Main1}\int_{H^n}|T_M \psi_r(x)|^2\rho(x)^{[\frac{n+3}{2}]}dx \leq Cr^{2[\frac{n+3}{2}]-(n+1)},\quad 0<r<\infty.\ee
We first need some results which will be useful for proving (\ref{eq:Main1}).
\begin{thm}
If $f$ is a Schwartz class function in $H^n$, then we have
$$\widehat{(itf)}(\lambda)=\Theta(\lambda)\hat{f}(\lambda).$$
\end{thm}
\begin{proof}For $\alpha, \beta\in \N$, let us consider the special Hermite functions on $\C^n$ defined as follows
$$\Phi^\lambda_{\alpha,\beta}(z)=(2\pi)^{-\frac{n}{2}}(\pi_\lambda(z)\phi^\lambda_\alpha, \phi^\lambda_\beta)$$
where $\phi_\alpha^\lambda$ are Hermite functions and $\pi_\lambda(z)=\pi_\lambda(z,0)$. Then using \cite[Proposition 1.3.2]{T} we can easily see that
\be\label{eq:dM2}|\lambda|^nW_\lambda(\Phi^\lambda_{\alpha+m^+,\alpha+m^-})=(2\pi)^\frac{n}{2} (-1)^{|m^+|}V^m_\alpha(\lambda).\ee
Let $f$ be a Schwartz class function on $H^n$. Then for any $\lambda\in \R^*$ we can write
$$ f^\lambda=\sum_{m, \alpha} B(\lambda, m,\alpha)|\lambda|^n\Phi^\lambda_{\alpha+m^+, \alpha+m^-}.$$
We now calculate $\frac{\partial}{\partial \lambda}f^\lambda$. Using the relation $\Phi^\lambda_{\alpha+m^+, \alpha+m^-}(z)=\Phi_{\alpha+m^+, \alpha+m^-}(\sqrt{\lambda}z)$ one can easily see that
\begin{align*}
\frac{\partial}{\partial \lambda}f^\lambda&= \sum_{m,\alpha} \frac{\partial}{\partial\lambda}B(\lambda,m,\alpha)|\lambda|^n\Phi^\lambda_{\alpha+m^+, \alpha+m^-}+\frac{n}{\lambda}\sum_{m,\alpha} B(\lambda,m,\alpha)|\lambda|^n\Phi^\lambda_{\alpha+m^+, \alpha+m^-}\\
&\quad +\frac{1}{2\lambda}\sum_{m,\alpha} B(\lambda,m,\alpha)|\lambda|^n\sum_{j=1}^n\left((z_j\frac{\partial}{\partial z_j}+\bar{z_j}\frac{\partial}{\partial \bar{z_j}})\Phi_{\alpha+m^+, \alpha+m^-}\right)(\sqrt{\lambda}\,\cdot).
\end{align*}

From \cite[(1.3.17), (1.3.18)]{T} and \cite[(1.3.23)]{T} we have
\begin{align*}
& 2 z_j \frac{\partial}{\partial z_j}\Phi_{\alpha+m^+,\alpha+m^-} \\
&= \frac{i}{2} z_j\big[(2(\alpha_j+m^+_j)+2)^{\frac{1}{2}} \Phi_{\alpha+m^++e_j,\alpha+m^-}+(2(\alpha_j+m^-_j))^{\frac{1}{2}}\Phi_{\alpha, \alpha+m^-- e_j}\big]\\
&= -\frac{1}{2}\big[ (2(\alpha_j+m^+_j)+2)\Phi_{\alpha+m^+,\alpha+m^-}\\
&\quad -(2 (\alpha_j+m^+_j)+2)^{\frac{1}{2}} (2 (\alpha_j+m^-_j)+2)^{\frac{1}{2}} \Phi_{\alpha+m^++e_j,\alpha+m^-+e_j}\big]\\
&\quad -\frac{1}{2}\big[(2(\alpha_j+m^+_j))^{\frac{1}{2}} (2(\alpha_j+m^-_j))^{\frac{1}{2}} \Phi_{\alpha+m^+-e_j,\alpha+m^--e_j}-(2(\alpha_j+m^-_j))\Phi_{\alpha+m^+,\alpha+m^-}\big]\\
&= (\alpha_j+m^+_j+1)^{\frac{1}{2}} (\alpha_j+m^-_j+1)^{\frac{1}{2}}\Phi_{\alpha+m^++e_j, \alpha+m^-+e_j}\\
&\quad -(\alpha_j+m^+_j)^{\frac{1}{2}} (\alpha_j+m^-_j)^{\frac{1}{2}} \Phi_{\alpha+m^+-e_j,\alpha+m^--e_j}+(m_j -1)\Phi_{\alpha+m^+,\alpha+m^-}.
\end{align*}
Similarly, we can also show that
\begin{align*}
2\bar{z}_j \frac{\partial}{\partial \bar{z}_j} \Phi_{\alpha+m^+,\alpha+m^-}&= (\alpha_j+m^+_j+1)^{\frac{1}{2}} (\alpha_j+m^-_j+1)^{\frac{1}{2}} \Phi_{\alpha+m^++e_j,\alpha+m^-+e_j} \\&\quad -(\alpha_j+m^+_j)^{\frac{1}{2}} (\alpha_j+m^-_j)^{\frac{1}{2}}  \Phi_{\alpha+m^+-e_j, \alpha+m^--e_j} \\& \quad-(m_j+1) \Phi_{\alpha, \alpha+m^-}.\end{align*}
Combining the above two relations we get
\begin{align*}
\Big(z_j\frac{\partial}{\partial z_j}+\bar{z}_j\frac{\partial}{\partial \bar{z}_j}\Big) \Phi_{\alpha+m^+,\alpha+m^-}&= ((\alpha_j+m^+_j)+1)^{\frac{1}{2}} ((\alpha_j+m^-_j)+1)^{\frac{1}{2}} \Phi_{\alpha+m^++e_j,\alpha+m^-+e_j}\\&\quad- (\alpha_j+m^+_j)^{\frac{1}{2}} (\alpha_j+m^-_j)^{\frac{1}{2}}  \Phi_{\alpha+m^+-e_j, \alpha+m^--e_j} - \Phi_{\alpha, \alpha+m^-}.\end{align*}


So, we have
\begin{align*}
\frac{\partial}{\partial \lambda}f^\lambda&=\sum_{m,\alpha} \frac{\partial}{\partial\lambda}B(\lambda,m,\alpha)|\lambda|^n\Phi^\lambda_{\alpha+m^+, \alpha+m^-}+\frac{n}{2\lambda}\sum_{m,\alpha} B(\lambda,m,\alpha)|\lambda|^n\Phi^\lambda_{\alpha+m^+, \alpha+m^-}\\&\quad +\frac{1}{2\lambda}\sum_{m,\alpha} B(\lambda,m,\alpha-e_j)\sqrt{\alpha_j(\alpha_j+|m_j|)}|\lambda|^n\Phi^\lambda_{\alpha+m^+, \alpha+m^-}\\
&\quad - \frac{1}{2\lambda}\sum_{m,\alpha} B(\lambda,m,\alpha+e_j)\sqrt{(\alpha_j+1)(\alpha_j+|m_j|+1)}|\lambda|^n\Phi^\lambda_{\alpha+m^+, \alpha+m^-}.\end{align*}
Using (\ref{eq:dM1}) and (\ref{eq:dM2}), we get our required result from the above equation.
\end{proof}

For any function $b$ defined on $\R$, Let us consider the function
$$b(L_\lambda)=\sum_{k=0}^\infty b((2k+n)|\lambda|)|\lambda|^n \varphi^{n-1}_{k,\lambda}.$$
Here $L_\lambda$ stands for special Hermite operators with parameter $\lambda$. Then we have the following corollary.

\begin{cor}\label{lem:lemma1}For any Schwarz class function $b$ defined on $\R$, it holds
\begin{align*} \frac{d}{d\lambda}b(L_\lambda)&= \sum_{k=0}^\infty (2k+n)b'((2k+n)|\lambda)|)|\lambda|^n \varphi^{n-1}_{k,\lambda}\\
&\quad- \frac{1}{2\lambda}\sum_{k=0}^\infty k \Delta_-b((2k+n)|\lambda|)|\lambda|^n \varphi^{n-1}_{k,\lambda}\\
&\quad-\frac{1}{2\lambda}\sum_{k=0}^\infty (k+n) \Delta_+ b((2k+n)|\lambda|)|\lambda|^n \varphi^{n-1}_{k,\lambda}.\end{align*}
\end{cor}

\begin{proof} As $\varphi^{n-1}_{k,\lambda}=\sum_{|\alpha|=k}\Phi^{\lambda}_{\alpha,\alpha}$ from the proof of the above theorem we have
\begin{align*}
\label{eq:lem1} \frac{d}{d\lambda}b(L_\lambda)&= \sum_{k=0}^\infty (2k+n)b'((2k+n)|\lambda|)|\lambda|^n \varphi^{n-1}_{k,\lambda}+
\frac{n}{2\lambda}\sum_{k=0}^\infty b((2k+n)|\lambda|)|\lambda|^n \varphi^{n-1}_{k,\lambda}\\
&\quad +\frac{1}{2\lambda}\sum_{k=0}^\infty b((2k-2+n)|\lambda|)|\lambda|^n\sum_{|\alpha|=k}\sum_{j=1}^n \alpha_j \Phi_{\alpha,\alpha}(\sqrt{|\lambda|}z)\\&\quad -
\frac{1}{2\lambda}\sum_{k=0}^\infty b((2k+2+n)|\lambda|)|\lambda|^n\sum_{|\alpha|=k}\sum_{j=1}^n (\alpha_j+1) \Phi_{\alpha,\alpha}(\sqrt{|\lambda|}z)\\
&=\sum_{k=0}^\infty (2k+n)b'((2k+n)|\lambda)|)|\lambda|^n \varphi^{n-1}_{k,\lambda}- \frac{1}{2\lambda}\sum_{k=0}^\infty k \Delta_-b((2k+n)|\lambda|)|\lambda|^n \varphi^{n-1}_{k,\lambda}\\&\quad -\frac{1}{2\lambda}\sum_{k=0}^\infty (k+n) \Delta_+ b((2k+n)|\lambda|)|\lambda|^n \varphi^{n-1}_{k,\lambda}.
\end{align*}
Hence the corollary is proved.
\end{proof}
The next lemma is similar to  \cite[Lemma 2]{L}.
\begin{lem}\label{lem:estimate}
The following inequality holds
\be\label{eq:estimate}|\lambda|^{-\frac{|\alpha|+|\beta|}{2}}\|\chi_{N}(\lambda) \delta^\alpha(\lambda) \bar{\delta}^\beta(\lambda)W_\lambda\big(\frac{\partial^l}{\partial\lambda^l}\psi_r^\lambda\big)\|_{\operatorname{op}}\leq C~2^{-N(l+\frac{|\alpha|+|\beta|}{2})}f_{\alpha, \beta,l}(r2^N)\ee
for all $|\alpha|+|\beta|+2l\leq 2 [\frac{n+3}{2}]$, where $f_{\alpha, \beta,l}$ is a rapidly decreasing function.
\end{lem}
\begin{rem}
 As we have mentioned earlier, in \cite{L} Lin gave a detailed proof of Lemma~2 only for some very particular type of polynomials of the form $P(z,t)=z_1^a \bar{z_1}^b$ with $a\geq b$ and $P(z,t)=z_1^a \bar{z_2}^b$ and that also involved very long and technical calculations. Now, we have already discussed in the Section \ref{sec:preli} that $\frac{1}{\sqrt{|\lambda|}}\delta_j(\lambda)$ and $\frac{1}{\sqrt{|\lambda|}}\bar{\delta_j}(\lambda)$ are similar to the operators $\Delta_{z_j}$ and $\Delta_{\bar{z}_j}$ defined in \cite{L}, respectively. So, we can now handle the case associated with the operators $|\lambda|^{-\frac{|\alpha|+|\beta|}{2}}\delta^\alpha(\lambda)\bar{\delta}^\beta(\lambda)$ very easily by using Lemma \ref{lem:lambda version}. It turns out to be a little more difficult when we have to consider the operator involving $\Theta(\lambda)$ also.
\end{rem}
\begin{proof}[Proof of Lemma \ref{lem:estimate}]
We will prove this lemma only for $\lambda>0$. The other case can be done similarly. Let us consider the function $$b_r (x)= e^{-\frac{r}{2}x}-e^{-r x}.$$
Then, from Lemma \ref{lem:lemma1} we can see that the Hermite coefficient of $W_\lambda\left(\frac{\partial}{\partial\lambda}b_r(L_\lambda)\right)$ is
$$\Big(\frac{\partial}{\partial\lambda}-\frac{1}{2\lambda}k\Delta_--\frac{1}{2\lambda}(k+n)\Delta_+\Big)b_r((2k+n)\lambda).$$
For convenience we use the notation $\Gamma^k_\lambda$ in place of $\left(\frac{\partial}{\partial\lambda}-\frac{1}{2\lambda}k\Delta_--\frac{1}{2\lambda}(k+n)\Delta_+\right)$. Hence, $\Gamma^k_\lambda b_r((2k+n)\lambda)$ equals to
\begin{align*}
&(2k+n)b_r'((2k+n)\lambda)-k\int_0^1b_r'((2(k-u)+n)\lambda)du\\
&\quad -(k+n)\int_0^1b_r'((2(k+u)+n)\lambda)du\\
&=-\frac{1}{2}(2k+n)\lambda\int^1_0\int_0^u b_r''((2(k-v)+n)\lambda)du\\
&\quad+\frac{1}{2}(2k+n)\lambda\int_0^1\int_0^u b_r''(2(k+v)+n)\lambda)dv\\
&\quad +\frac{n}{2}\int_0^1 b_r'((2(k-u)+n)\lambda)du
-\frac{n}{2}\int_0^1 b_r'((2(k+u)+n)\lambda)du.
\end{align*}
Let $Q\subset \R^2$ be the set enclosed by the three lines namely $x=0$, $y=x$ and $y=1$ and $\sigma_1$, $\sigma_2$ be the two functions on $Q$ defined by $\sigma_1(u_,v)=-v$ and $\sigma_2(u,v)=v$. Also let $\varsigma_1(w)=-w$ and $\varsigma_2(w)=w$, where $w\in [0,1]$. Then the Hermite coefficient of $\frac{\partial}{\partial\lambda}b(H(\lambda))$ can be written as
\begin{align*}
 &-\frac{1}{2}(2k+n)\lambda\int_Q b_r''((2(k+\sigma_1(u,v))+n)\lambda)dudv\\
 &\quad + \frac{1}{2}(2k+n)\lambda\int_Q b_r''(2(k+\sigma_2(u,v))+n)\lambda)dudv\\
 &\quad +\frac{n}{2}\int_0^1 b_r'((2(k+\varsigma_1(w))+n)\lambda)dw
\\&\quad- \frac{n}{2}\int_0^1 b_r'((2(k+\varsigma_2(w))+n)\lambda)dw.
\end{align*}
Now let $Q^m$ be the cartesian product of $m$ copies of $Q$. Let us use the notation $(u,v) $, for any element $(u_1,v_1, \cdots, u_m,v_m)$ of $Q^m$. Then, we claim that the Hermite coefficient of $\frac{\partial^l}{\partial\lambda^l}b_r(H(\lambda))$ can be written as the sum of several terms (the number of terms appearing in the sum depends only on $n$ and $l$) which are of the form
$$(2k+n)^m \lambda^m \int_{Q^m}\int_{[0,1]^{l-m}} g(u,v,w)b_r^{l+m}\left(\left(2\left(k+\sigma(u,v)+\varsigma(w)\right)+n\right)\lambda\right)dudv dw,$$
where $0\leq m\leq l$. Here $g$, $\sigma$ and $\varsigma$ are some bounded functions on $Q^m\times [0,1]^{l-m}$, $Q^m$ and $[0,1]^{l-m}$ respectively where the bounds depend only on $n$ and $l$. Also, $b^{l+m}_r$ stands for $(l+m)$-th derivative of $b_r$.

We will prove our claim using induction on $l$. We already show that the result is true for $l=1$. Suppose the result is true for some $l\in \N$. Since, $\Gamma^k_\lambda$ satisfies the Leibnitz rule and $\Gamma^k_\lambda((2k+n)\lambda)$ vanishes, the Hermite coefficient of $\frac{\partial^{l+1}}{\partial\lambda^{l+1}}b_r(H(\lambda))$ can be written as the sum of several terms which are of the form
\Bea \int_{Q^m}\int_{[0,1]^{l-m}} (2k+n)^m \lambda^m g(u,v,w)\Gamma^k_\lambda
b_r^{l+m}\left(\left(2\left(k+\sigma(u,v)+\varsigma(w)\right)+n\right)\lambda\right)dudvdw
\Eea
where $0\leq m\leq l$ and $C$ is some constant depending only on $n$ and $l$. Also $g$, $\sigma$ and $\varsigma$ are bounded functions on $Q^m\times [0,1]^m$, $Q^m$ and $[0,1]^{l-m}$ respectively. The above expression can be written as the sum of two terms $I_1$ and $I_2$ where
 \begin{align*}
 I_1&=\int_{Q^m\times [0,1]^{l-m}}\big(2k+n\big) (2k+n)^m \lambda^m g(u,v,w)b_r^{l+m+1}\\
 &\qquad \times\big((2(k+\sigma(u,v)+\varsigma(w))+n)\lambda\big)dudvdw\\
 &\quad -
 \int_{Q^m\times [0,1]^{l-m}} k(2k+n)^m \lambda^m \\
 &\qquad \times \int_0^1g(u,v,w)b_r^{l+m+1}((2(k+u')+2\sigma(u,v)+2\varsigma(w)+n)\lambda)du'dudvdw\\
 &\quad - \int_{Q^m\times[0,1]^{l-m}} (k+n)(2k+n)^m \lambda^m\\&\qquad \times\int_0^1g(u,v,w)b_r^{l+m+1}((2(k-u')+2\sigma(u,v)+2\varsigma(w)+n)\lambda)du'dudvdw.
 \end{align*}
and
\begin{multline*}
I_2=\int_{Q^m\times [0,1]^{l-m}} (2k+n)^m \lambda^m 2\big(\sigma(u,v)+\varsigma(w)\big)g(u,v,w)b_r^{l+m+1}\\
 \times\big((2(k+\sigma(u,v)+\varsigma(w))+n)\lambda\big)dudvdw.
\end{multline*}
$I_1$ can be dealt with similarly as in the case $l=1$. On the other hand, $I_2$ can be written as
\begin{multline*}
(2k+n)^m \lambda^m \int_{Q^m}\int_{[0,1]^{l-m}} \int_{[0,1]}\tilde{g}(u,v,w,w')b_r^{l+m+1}\\
\times\big((2(k+\sigma(u,v)+\tilde{\varsigma}(w,w'))+n)\lambda\big)dudv dwdw',
\end{multline*}
where $\tilde{g}(u,v,w,w')=2\big(\alpha(u,v)+\beta(w)\big)g(u,v,w)$ and $\tilde{\varsigma}(w,w')=\varsigma(w)$. Hence our claim is proved.

Now we are in a position to estimate $\delta^\alpha(\lambda) \bar{\delta}^\beta(\lambda)W_\lambda(\frac{\partial^l}{\partial\lambda^l}\psi_r^\lambda)$. From the above discussion and using Lemma \ref{lem:lambda version} one can notice that $\delta^\alpha(\lambda) \bar{\delta}^\beta(\lambda)W_\lambda(\frac{\partial^l}{\partial\lambda^l}\psi_r^\lambda)$ can be written as sum of several operators of the form
$$ \int_{Q^m}\int_{[0,1]^{l-m}} g(u,v,w)\lambda^{-\frac{|\beta|+2|\gamma|-|\alpha|}{2}}\left(A^*(\lambda)\right)^{\alpha+\gamma-\beta}A^\gamma(\lambda)D^{|\gamma|}_-D^{|\beta|}_+\tilde{b}^{l,m}_{r,u,v}(H(\lambda))dudv dw,$$
where $\gamma\in \N^n$ satisfying $0\leq \gamma\leq \alpha \leq \beta+\gamma$ and $\tilde{b}^{l,m}_{r,u,v}(k,\lambda)= (2k+n)^m \lambda^m b_r^{l+m}\big((2(k+\sigma(u,v)+\varsigma(w))+n)\lambda\big)$. Of course, the total number of term depends only on $n, p, q$ and $l$.
Therefore, in order to prove our lemma we only need to estimate the operator norm of the operators which are of the form
\begin{multline}
\label{int op}
\int_{Q^m}\int_{[0,1]^{l-m}} g(u,v,w)\lambda^{-\frac{|\beta|+2|\gamma|-|\alpha|}{2}}\chi_N(\lambda)\big(A^*(\lambda)\big)^{\alpha+\gamma-\beta}A^\gamma(\lambda)\\
\times D^{|\gamma|}_-D^{|\beta|}_+\tilde{b}^{l,m}_{r, u,v}(H(\lambda))dudv dw.\end{multline}
Since the finite difference operators can be estimated by derivatives, we have
$$|D^{|\gamma|}_-D^{|\beta|}_+\tilde{b}^{l,m}_{r,u,v}(k,\lambda))|\leq C~ |\partial^s_1\tilde{b}^{l,m}_{r, u,v}(k,\lambda)|,$$
where $|\gamma|+|\beta|\leq [\frac{n+3}{2}]$ and $\partial^s_1$ stands for the partial derivative of order $s$ with respect to the first variable. The above can be further dominated by
\be\label{est:b}\sum_{i=0}^s C_{r,m,l} (2k+n)^{m-i}\lambda^{m+s-i} |b_r^{l+m+s-i}((2k+n+\sigma(u,v)+\varsigma(w))\lambda)|.\ee
Now, if $(2k+n)\lambda\sim 2^N$ and $l+m=\vartheta$, it can be shown (\cite[Lemma 2.2]{BT}) that
$$|b_r^{l+m+s-i}((2k+n+\sigma(u,v)+\varsigma(w))\lambda)|\lesssim 2^N r 2^{-N(\vartheta+s-i)}\tilde{f}_{\vartheta, s, i}(2^N r)$$
where $\tilde{f}_{\vartheta,s,i}(x)=x^{\vartheta+s-i} e^{-cx} +x^{\vartheta+s-i-1}e^{-cx}$, a rapidly decreasing function for each $i$.
Hence, (\ref{est:b}) is bounded by a constant multiple of
$$2^N r \lambda^s2^{-N(l+s)}\tilde{f}_{l, s}(2^N r),$$
where $\tilde{f}_{l,r}$ is a rapidly decreasing function.

Recall that $A^*_j(\lambda)\Phi_\mu^\lambda=(2\mu_j+2)^{\frac{1}{2}}\lambda^\frac{1}{2}\Phi_{\mu+e_j}^\lambda$ and $A_j(\lambda)\Phi_\mu^\lambda=(2\mu_j)^{\frac{1}{2}}\lambda^\frac{1}{2}\Phi_{\mu-e_j}^\lambda$. Hence, using the boundedness of the function $g$, the operator norm of (\ref{int op}) can be dominated by
 $$\lambda^{-\frac{|\beta|+2|\gamma|-|\alpha|}{2}}2^{N\frac{|\alpha|+2|\gamma|-|\beta|}{2}}2^N r \lambda^{|\gamma|+|\beta|}2^{-N(l+|\gamma|+|\beta|)}\tilde{f}_{l, r}(2^Rt_{j+1}),$$
 which is equal to
 $$\lambda^{\frac{|\alpha|+|\beta|}{2}}2^{-N(l+\frac{|\alpha|+|\beta|}{2})}f_{\alpha,\beta,l}(r2^N),$$
 where $f_{\alpha,\beta,l}$ is a rapidly decreasing function. Hence the lemma is proved.
\end{proof}
Now we are in a position to prove Theorem \ref{th:Main}. As we have discussed earlier we only have to prove (\ref{eq:Main1}).
\begin{lem}\label{lem:main2} For $l\in\R$ satisfying $0\leq l\leq [\frac{n+3}{2}] $, we have
$$\int_{H^n}|T_M\psi_r(z,t)|^2 \rho(z,t)^ldzdt\lesssim r^{2l-(n+1)}.$$
\end{lem}
\begin{proof} We will prove this lemma only for $l\in \N$. For other $l$ the estimate can be obtained easily by using the estimate of $[l]$ and $[l+1]$.

As $\rho(z,t)^l\lesssim (\sum_{i=1}^n |z_i|^2)^{2l}+ t^{2l}$, we will prove that
\be\label{estim:main2}\int_{H^n}(\sum_{i=1}^n|z_i|^2)^{2l} |T_M\psi_r(z,t)|^2dzdt\lesssim ~r^{2l-(n+1)}\ee
and
\be\label{estim:main3}\int_{H^n}t^{2l} |T_M\psi_r(z,t)|^2dzdt\lesssim ~r^{2l-(n+1)}.\ee
We first prove the estimate (\ref{estim:main3}). Using the Plancherel theorem on the $t$-variable we can observe that the left hand side of (\ref{estim:main3}) equals to a constant multiple of
$$\int_{\R}\|\frac{\partial^l}{\partial \lambda^l} T^\lambda_{M(\lambda)}\psi_r^\lambda\|^2_2d\lambda.$$
Thus by Leibnitz rule, it is enough to prove the following inequality
$$\int_{\R} \|\Big( \frac{\partial^{l-l_1}}{\partial \lambda^{l-l_1}} T^\lambda_{M(\lambda)}\Big)
\Big(\frac{\partial^{l_1}}{\partial \lambda^{l_1}} \psi^\lambda_r\Big)\|^2_2 d\lambda\lesssim ~r^{2l-(n+1)}$$
for $0\leq l_1\leq l$. Applying Lemma \ref{lem:WZ}, the left hand side of the above inequality can be dominated by
 \be\label{estim:main3i}\sum_{|a|+|b|+|c|+|d|+2s=l-l_1} C_{a,b, c,d,s}\int_{\R}\Big\|\frac{1}{|\lambda|^{\frac{|a|+|b|}{2}}}
 T_{\delta^a(\lambda)\bar{\delta}^b(\lambda)\Theta^s(\lambda) M(\lambda)}\Big(z^c \bar{z}^d\frac{\partial^{l_1}}{\partial \lambda^{l_1}}\psi^\lambda_r\Big)\Big\|^2_2d \lambda.\ee
 Therefore, using Plancherel theorem, it is enough to estimate
 \be\label{estim:main3ii}\int_{\R}\frac{1}{|\lambda|^{|\alpha|+|\beta|+|c|+|d|}}\Big\|\delta^a(\lambda) \bar{\delta}^b(\lambda)
 \Theta^s(\lambda) M(\lambda)\delta^c(\lambda) \bar{\delta}^d(\lambda)
W_\lambda(\frac{\partial^{l_1}}{\partial \lambda^{l_1}}\psi^\lambda_r)\Big\|^2_{\operatorname{HS}} |\lambda|^{n}d\lambda,\ee
 where $|a|+|b|+|c|+|d|+2s=l-l_1$.
 Now we can write
 \begin{multline*}
 \delta^a(\lambda) \bar{\delta}^b(\lambda)\Theta^s(\lambda)
 M(\lambda)\delta^c(\lambda) \bar{\delta}^d(\lambda)W_\lambda\Big(\frac{\partial^{l_1}}{\partial \lambda^{l_1}}\psi^\lambda_r\Big)
\\=\sum_{N=0}^\infty \Big(\delta^a(\lambda)\bar{\delta}^b(\lambda)\Theta^s(\lambda)M(\lambda)\Big)\chi_N(\lambda)\cdot
 \chi_N(\lambda)\Big(\delta^c(\lambda) \bar{\delta}^d(\lambda)W_\lambda(\frac{\partial^{l_1}}{\partial \lambda^{l_1}}\psi^\lambda_r)\Big)\end{multline*}
 So, (\ref{estim:main3ii}) can be dominated by
 \begin{multline}
 \label{estim:main5}
 \sum_{N=0}^\infty\int_{\R}\| \lambda^{-\frac{|a|+|b|}{2}}\delta^a(\lambda)\bar{\delta}^b(\lambda)\Theta^s(\lambda)M(\lambda)\chi_N(\lambda)\|^2_{\operatorname{HS}}\\
\times \|\lambda^{-\frac{|c|+|d|}{2}}\chi_N(\lambda)\delta^c(\lambda) \bar{\delta}^d(\lambda)W_\lambda\Big(\frac{\partial^{l_1}}{\partial \lambda^{l_1}}\psi^\lambda_r\Big)\|^2_{\operatorname{op}}|\lambda|^{n}d\lambda.
 \end{multline}
 By Lemma \ref{lem:estimate}, we have
 $$\|\chi_N(\lambda)\delta^c(\lambda) \bar{\delta}^d(\lambda)W_\lambda\Big(\frac{\partial^{l_1}}{\partial \lambda^{l_1}}\psi^\lambda_r\Big)\|^2_{\operatorname{op}}\lesssim |\lambda|^{(|c|+|d|)} 2^{-N(|c|+|d|+2l_1)} f_{c,d,l_1}(2^N r),$$
 where $f_{c,d,l_1}$ is a rapidly decreasing function.
 Using the above estimate and the hypothesis of Theorem \ref{th:Main} we get
 \begin{align*}
 \int_{\R}\|\lambda^{-\frac{|a|+|b|+|c|+|d|}{2}}\delta^a(\lambda) \bar{\delta}^b(\lambda)
 &\Theta^s(\lambda) M(\lambda)\delta^c(\lambda) \bar{\delta}^d(\lambda)
W_\lambda\Big(\frac{\partial^{l_1}}{\partial \lambda^{l_1}}\psi^\lambda_r\Big)\|^2_{\operatorname{HS}} |\lambda|^{n}d\lambda
\\
&\lesssim \sum_{N=0}^\infty 2^{N(n+1-|a|-|b|-2s)} 2^{-N(|c|+|d|+2l_1)}f_{c,d,l_1}(2^Nt_{j+1})\\
&\lesssim \sum_{N=0}^\infty 2^{N(n+1-2l)}f_{c,d,l_1}(2^N r)\\
&\lesssim r^{2l-(n+1)}.
 \end{align*}
We can estimate (\ref{estim:main3}) similarly by observing the fact that it will lead to the estimate of the integral (\ref{estim:main3ii}) with $s=l_1=0$.
\end{proof}

We will conclude this section with the following lemma which will be used in Section~\ref{sec:proofm} in order to prove Theorem \ref{th:Main1}.

\begin{lem}\label{lem:2}
\begin{enumerate}Let $0<r<1$. For $l\leq [\frac{n+1}{2}]$ and $i=1,2,\cdots, n,$ the following estimates are true.
\item $\int_{H^n}|\partial_tT_M\psi_r(z,t)|^2 \rho(z,t)^ldzdt\lesssim r^{2l-n-3}$.\\
\item $\int_{H^n}|X_iT_M\psi_r(z,t)|^2 \rho(z,t)^ldzdt\lesssim r^{2l-n-3}$.\\
\item $\int_{H^n}|Y_iT_M\psi_r(z,t)|^2 \rho(z,t)^ldzdt\lesssim r^{2l-n-3}$.\end{enumerate}
\end{lem}
\begin{proof}
We will first prove (1). Similar to Lemma \ref{lem:main2}, here also we have to estimate
\be\label{estim:main4}\int_{H^n}\Big(\sum_{i=1}^n|z_i|^2\Big)^{2l} |\partial_t T_M\psi_r(z,t)|^2dzdt\lesssim ~r^{2l-n-3}\ee
and
\be\label{estim:main5A}\int_{H^n}t^{2l} |\partial_t T_M\psi_r(z,t)|^2dzdt\lesssim ~r^{2l-n-3}.\ee
We only prove (\ref{estim:main5A}). The estimate (\ref{estim:main4}) can be proved similarly. By Plancherel theorem in the $t$- variable, the left hand side of (\ref{estim:main5}) is a constant multiple of
$$\int_{\C^n}\|\frac{\partial^l}{\partial\lambda^l}\big(\lambda T^\lambda_{M(\lambda)}\big)\psi_r^\lambda\|^2_2d\lambda.$$
To estimate the above integral it is enough to estimate the following two integrals
$$\int_{\C^n}\|\lambda\frac{\partial^l}{\partial\lambda^l} T^\lambda_{M(\lambda)}\psi_r^\lambda\|^2_2d\lambda$$ and $$\int_{\C^n}\|\frac{\partial^{l-1}}{\partial\lambda^{l-1}} T^\lambda_{M(\lambda)}\psi_r^\lambda\|^2_2d\lambda .$$
Using Lemma \ref{lem:main2}, we can see that the second term satisfies our required estimate. So, we only have to prove the same for the first term. Let us observe that if $2^N\leq (2k+n)|\lambda|< 2^{N+1}$, then $|\lambda|< 2^{N+1}$. So,
\begin{align*}
&\int^\infty_{-\infty}|\lambda|^2\|\lambda^{-\frac{\alpha+\beta}{2}}\delta^\alpha(\lambda)
\bar{\delta}^\beta(\lambda)\Theta^s(\lambda)M(\lambda)\chi_{N}(\lambda)\|_{\operatorname{HS}}^2|\lambda|^n d\lambda\\
&\lesssim 2^{N}\int^{2^{N+1}}_{-2^{N+1}}\|\lambda^{-\frac{\alpha+\beta}{2}}\delta^\alpha(\lambda)
\bar{\delta}^\beta(\lambda)\Theta^s(\lambda)M(\lambda)\chi_{N}(\lambda)\|_{\operatorname{HS}}^2|\lambda|^n d\lambda\\
&\lesssim 2^{N}\cdot2^{N(n+1-l)}< 2^{N(n+2-l)}.
\end{align*}
Hence, we can get our required estimate by proceeding similarly as in Lemma \ref{lem:main2} and using the fact that $r<1$.

Now we will prove (2). Again we will only estimate
\be\label{estim:main6}\int_{H^n}t^{2l} |X_i T_M\psi_r(z,t)|^2dzdt\lesssim ~r^{2l-n-2}.\ee
By the Plancherel theorem and the relation between $X_i$, $Z_i(\lambda)$ and $\bar{Z_i}(\lambda)$, it is enough to estimate
\be\label{estim:Z}\int_{\C^n}\|\frac{\partial^l}{\partial\lambda^l}\Big(Z_i(\lambda) T^\lambda_{M(\lambda)}\psi_r^\lambda\Big)\|^2_2d\lambda \ee
and
\be\int_{\C^n}\|\frac{\partial^l}{\partial\lambda^l}\Big(\bar{Z_i}(\lambda) T^\lambda_{M(\lambda)}\psi_r^\lambda\Big)\|^2_2d\lambda. \ee
Since both can be estimated similarly, we will only estimate (\ref{estim:Z}). Now (\ref{estim:Z}) can be dominated by
\be\label{exp:Z}\int_{\C^n}\|Z_i(\lambda)\frac{\partial^l}{\partial\lambda^l}\Big( T^\lambda_{M(\lambda)}\psi_r^\lambda\Big)\|^2_2d\lambda+\frac{1}{4}\int_{\C^n}\|z_i\frac{\partial^{l-1}}{\partial\lambda^{l-1}}\Big( T^\lambda_{M(\lambda)}\psi_r^\lambda\Big)\|^2_2d\lambda.\ee
We can estimate the second term similarly as (\ref{estim:main2}).

 In order to estimate the first term of (\ref{exp:Z}) notice that it is enough to estimate (see Lemma \ref{lem:main2}) the term
 $$\int_{\R}|\lambda|^{-(|a|+|b|)}\|Z_i(\lambda)
 T_{\delta^a(\lambda)\bar{\delta}^b(\lambda)\Theta^s(\lambda) M(\lambda)}\Big(z^c \bar{z}^{d}\frac{\partial^{l_1}}{\partial \lambda^{l_1}}\psi^\lambda_r\Big)\|^2_2d \lambda,$$
 where $|\alpha|+|\beta|+|\gamma|+|\delta|+2s=l-l_1$.

 Since, $W_\lambda(Z_i(\lambda)f)=i W_\lambda(f)A_j(\lambda)$, one can dominate this by
 \begin{multline}
 \label{estim:Z3}
 \sum_{N=0}^\infty\int_{\R}|\lambda|^{-(|a|+|b|+|c|+|d|)}\| \delta^a(\lambda)\bar{\delta}^b(\lambda)\Theta^s(\lambda)M(\lambda)\chi_N(\lambda)\|^2_{\operatorname{HS}}\\
 \times\|\chi_N(\lambda)\delta^c(\lambda) \bar{\delta}^d(\lambda)W_\lambda(\frac{\partial^{l_1}}{\partial \lambda^{l_1}}\psi^\lambda_r)A_i(\lambda)\|^2_{\operatorname{op}}|\lambda|^{n-\frac{l-l_1}{2}}d\lambda.
 \end{multline}
 Since
 $$
 A_i(\lambda)\Phi_\alpha^\lambda=(2\alpha_j)^{\frac{1}{2}}|\lambda|^{\frac{1}{2}}\Phi^\lambda_{\alpha-e_j},
 $$
 we can estimate $\|\chi_N(\lambda)\delta^c(\lambda) \bar{\delta}^d(\lambda)W_\lambda(\frac{\partial^{l_1}}{\partial \lambda^{l_1}}\psi^\lambda_r)A_i(\lambda)\|^2_{\operatorname{op}}$ as in Lemma \ref{eq:estimate} to obtain
 $$\|\chi_N(\lambda)\delta^c(\lambda) \bar{\delta}^d(\lambda)W_\lambda\Big(\frac{\partial^{l_1}}{\partial \lambda^{l_1}}\psi^\lambda_r\Big)A_i(\lambda)\|^2_{\operatorname{op}}\leq C~2^{-N(2l+|\alpha|+|\beta|-1)} g_{\alpha, \beta, l}(r2^N).$$
 As we get an extra $2^N$ on the right hand side, (\ref{estim:Z3}) can be bounded by $r^{l-(n+2)}$, as in Lemma \ref{lem:main2}. Thus (\ref{estim:main6}) is proved.

The third part of the lemma is similar to (2).
\end{proof}

\section{Proof of Theorem \ref{th:Main1}}
\label{sec:proofm}
We will start this section with some definitions. We will first describe the dyadic Heisenberg cubes from \cite{C} and \cite{H2}.
\begin{thm}\label{Th:dyadic}There exists a collection of open sets $\mathcal{D}=\{Q^j_\alpha\subset H^n:j\in \Z^n, \alpha\in I_j\}$, and absolute constants $0<\eta<1$, $a>0$, $b>0$ and $\epsilon>0$ such that
\begin{enumerate}
\item $|H^n\setminus \cup_\alpha Q^j_\alpha|=0$, for all $j$.\\
\item For $l\geq j$ and any $\alpha$, $\beta$, either $Q^l_\alpha$ is contained in $Q^j_\beta$ or they do not intersect.\\
\item For each $j$, $\alpha$ there exists a $\beta$ such that
$$Q^{j+1}_\beta\subset Q^{j}_\alpha.$$
$Q^{j+1}_\beta$ is called a child of $Q^j_\alpha$.\\
\item For each $j$, $\alpha$ there exist unique $\beta$ such that
$$Q^j_\alpha\subset Q^{j-1}_\beta.$$
$Q^{j-1}_\beta$ is called the parent of $Q^{j}_\alpha$.
\item If $Q^j_\alpha$ is a child of $Q^{j-1}_\beta$, then
$$|Q^j_\alpha|\geq \epsilon |Q^{j-1}_\beta|$$

\item There is a point $(z_\alpha^j, t_\alpha^j)$ such that $B((z_\alpha^j, t^j_\alpha), \eta^j)\subset Q^j_\alpha\subset B((z_\alpha^j, t^j_\alpha), a\eta^j) $.

\end{enumerate}
We will call those $Q^j_\alpha$ as cubes of side lenght $\eta^j$ and center $(z_\alpha^j, t_\alpha^j)$. For any $\gamma>0$ the dilation of any cube is defined as follows
$$\gamma Q^j_\alpha= B((z_\alpha^j, t_\alpha^j), a\gamma \eta^j).$$

\end{thm}

A collection of cubes $\mathcal{S}$ in $H^n$ is said to be $ \eta$-sparse if there are sets $\{E_S \subset S:S\in \mathcal{S}\}$ which are pairwise disjoint and satisfy $|E_S|>\eta|S|$ for all $S\in \mathcal{S}$.
Corresponding to a dyadic grid $\mathcal{D}$ and a sparse family $\mathcal{S}$, for $1\leq r<\infty$, we can consider a sparse operator, which is defined as follows.
$$\mathcal{A}_{r,\mathcal{S}}f(z,t)=\sum_{Q\in \mathcal{S}}\Big(\frac{1}{|Q|}\int_{Q}|f|^r\Big)^\frac{1}{r}\chi_Q(z,t).$$

Let $t_j =2^{-j}$, $j\in \N$. Let us define the following operators
$$T_j f(z,t)= T\psi_{t_j}*f(z,t).$$
For $N\in \N$, consider
$$T^Nf=\sum_{j=1}^N T_jf.$$
 It is easy to see that $T^Nf$ tends to $Tf$ in $L^2$ as $N$ tends to $\infty$. Let $k_j$ and $K^N$ are the kernels of $T_j$ and $T^N$ respectively.

For any cube $Q\subset H^n$, let us consider the operators
$$ \mathcal{T}_{j,Q} f(z,t)=\Big(\int_{H^n\setminus 3Q}k_j((z,t)(w,s)^{-1})f(w,s)dt\Big)\chi_{Q}(z,t)$$
and
$$\mathcal{T}^N_Q f(z,t)=\sum_{j=1}^N \mathcal{T}_{j,Q} f(z,t).$$

Also, consider the operator
$$\mathcal{T}^{N*}f(z,t)=\sup_{Q\ni(z,t)}|\mathcal{T}^N_Q f(z,t)|.$$
We will prove the following theorem.

\begin{thm}\label{thm:Strong}Suppose $M$ satisfy the hypothesis of the Theorem \ref{th:Main1}. Then, for $f\in C_c^\infty(H^n)$ and $(z_0,t_0)\in H^n$
\be\label{estim:T*} \mathcal{T}^{N*}f(z_0,t_0)\leq C~\left(\Lambda(\mathcal{T}^N f)(z_0,t_0)+\Lambda_2f(z_0,t_0)\right)\ee
where the constant does not depend on $N$. Here $\Lambda$ stands for the the maximal function associated to the Heisenberg group. Also, we will have
\be\label{estim:T*2}\|\mathcal{T}^{N*}\|_{L^{2,\infty}(H^n)}\leq C~ \|f\|_{L^2(H^n)}.\ee

\end{thm}
\begin{proof} As $\mathcal{T}^N$ are uniformly bounded in $L^2(\R^n)$ and $\Lambda$ and $\Lambda_2$ both satisfy weak type (2,2) estimate, (\ref{estim:T*2}) is an easy consequence of (\ref{estim:T*}).
So, we will prove (\ref{estim:T*}). Fix  a cube $Q$ which contains $(z_0, t_0)$. Define $f_1=f\chi_{3Q}$ and $f_2=f-f_1$. Let us also consider $\phi\in C_0^\infty (H^n)$ supported in the homogeneous ball $\{(z,t):\rho(z,t)<1\}$ and satisfying $\phi((z,t))=1$ whenever $\rho(z,t)< \frac{1}{2}$. Define
\begin{align*}
K^N_1((z,t),(w,s))&:=K^N((z,t),(w,s))\phi((z,t)(w,s)^{-1})\\
K^N_2((z,t),(w,s))&:=K^N((z,t),(w,s))(1-\phi((z,t)(w,s)^{-1})),\\
k_{j,1}((z,t),(w,s))&:=k_j((z,t),(w,s))\phi((z,t)(w,s)^{-1})
\end{align*}
and
$$
k_{j,2}((z,t),(w,s)):=k_j((z,t),(w,s))(1-\phi((z,t)(w,s)^{-1})).
$$
Also, let $T^N_1$ and $T^N_2$ be the integral operators corrseponding to the kernels $K^N_1$ and $K^N_2$, respectively.

Since $f_2$ is supported outside $3Q$, we have
\begin{align*}
T^N(f_2)((z_0,t_0))&=\int_{H^n} K^N ((z_0,t_0)(w,s)^{-1}) f_2(w,s)dwds\\
&=\int_{H^n\setminus 3Q} K^N ((z_0,t_0)(w,s)^{-1}) f(w,s)dwds=\mathcal{T}^N_{Q}f(z_0,t_0).
\end{align*}

Now, let $(z,t)\in Q$. We will first prove that \be\label{estim:f1} | T^N_1 f_2(z,t)- T^N_1f_2(z_0, t_0)|\leq C \Lambda_2 f(z_0, t_0),\ee
where $C$ is independent of $N$. We will make use of the estimates obtained in Section \ref{sec:main} in order to get the above estimate.
\begin{multline*}
|T^N_1 f_2(z,t)- T^N_1f_2(z_0, t_0)|\\\lesssim  \int_{H^n\setminus 3Q} |K_1^N\big((z,t)(w,s)^{-1}\big)-K_1^N\big((z_0,t_0)(w,s)^{-1}\big)||f(w,s)|dwds.
\end{multline*}
Using Cauchy-Schwarz inequality the above term can be dominated by
\begin{multline*}
\Big(\int_{H^n\setminus 3Q} \rho\big((z,t)(w,s)^{-1}\big)^{\frac{n+\frac{3}{2}}{2}}|K_1^N\big((z,t)(w,s)^{-1}\big)-K_1^N\big((z_0,t_0)(w,s)^{-1}\big)|^2dwds\Big)^\frac{1}{2}\\
\times\bigg(\int_{H^n\setminus 3Q}\frac{|f(w,s)|^2}{\rho\big((z,t)(w,s)^{-1}\big)^{\frac{n+\frac{3}{2}}{2}}}dwds\bigg)^\frac{1}{2}.
\end{multline*}
We claim that
\begin{multline}
\label{estim:main31} \int_{H^n\setminus 3Q} \rho\big((z,t)(w,s)^{-1}\big)^{\frac{n+\frac{3}{2}}{2}}\big|K_1^N\big((z,t)(w,s)^{-1}\big)-K_1^N\big((z_0,t_0)(w,s)^{-1}\big)\big|^2dwds\\
\lesssim ~l(Q).
\end{multline}
If the claim is true, then we have
\begin{align*}
| \mathcal{T}_N^1 f_2(z,t)- \mathcal{T}_N^1f_2(z_0, t_0)|&\lesssim l(Q)^{\frac{1}{2}}\Big(\sum_{k=1}^\infty \int_{3^{k+1}Q\setminus 3^k Q}\frac{|f(w,s)|^2}{\rho\big((z,t)(w,s)^{-1}\big)^{\frac{n+\frac{3}{2}}{2}}}dwds\Big)^\frac{1}{2} \\
&\lesssim l(Q)^{\frac{1}{2}}\Big( \sum_{k=1}^\infty\frac{1}{(a 3^k l(Q))^{2(n+\frac{3}{2})}}\int_{3^{k+1}Q} |f(w,s)|^2dwds \Big)^\frac{1}{2}\\
&\lesssim \Big(\sum_{k=1}^\infty 3^{-k}\Big)^\frac{1}{2} \Lambda_2 f(z_0, t_0)\\
&\lesssim \Lambda_2 f(z_0,t_0).
\end{align*}
Hence  (\ref{estim:f1}) is proved.

 In order to prove (\ref{estim:main31}), it is enough to prove
 \begin{multline}
 \label{estim:min}
 \int_{H^n\setminus 3Q} \rho\big((z,t)(w,s)^{-1}\big)^{\frac{n+\frac{3}{2}}{2}}\big|k_{j,1}\big((z,t)(w,s)^{-1}\big)-k_{j,1}\big((z_0,t_0)(w,s)^{-1}\big)\big|^2dwds\\
 \lesssim ~l(Q)\min\Big\{\frac{t_{j+1}^\frac{1}{2}}{l(Q)}, \frac{l(Q)}{t_{j+1}^{\frac{1}{2}}}\Big\}.
 \end{multline}
 As $(z,t)\in Q$ and $(w,s)\in H^n\setminus  3Q$, $\rho\big((z_0, t_0)(w,s)^{-1}\big)$ and $\rho\big((z,t)(w,s)^{-1}\big)$ are comparable. So using Lemma \ref{lem:main2} we have
\begin{align*}
&\int_{H^n\setminus 3Q} \rho\left((z,t)(w,s)^{-1}\right)^{\frac{n+\frac{3}{2}}{2}}|k_{j,1}\left((z,t)(w,s)^{-1}\right)-k_{j,1}\left((z_0,t_0)(w,s)^{-1}\right)|^2dwds\\
&\lesssim \int_{H^n\setminus 3Q} \rho\left((z,t)(w,s)^{-1}\right)^{\frac{n+\frac{3}{2}}{2}}|k_{j,1}\left((z,t)(w,s)^{-1}\right)|^2dwds \\
&\lesssim t_{j+1}^{\frac{1}{2}}.
\end{align*}
Therefore, we only have to show that
\begin{multline*}
\int_{H^n\setminus 3Q} \rho\big((z_0,t_0)(w,s)^{-1}\big)^{\frac{n+\frac{3}{2}}{2}}\big|k_{j,1}\big((z,t)(w,s)^{-1}\big)-k_{j,1}\big((z_0,t_0)(w,s)^{-1}\big)\big|^2dwds\\
\lesssim ~\frac{l(Q)^2}{t_{j+1}^\frac{1}{2}}.
\end{multline*}
By change of variable, it is enough to prove
$$\int_{H^n\setminus 2Q_0} \rho(w,s)^{\frac{n+\frac{3}{2}}{2}}|k_{j,1}\big((z,t)(z_0, t_0)^{-1} (w,s)\big)-k_{j,1}(w,s)|^2dwds\lesssim ~\frac{l(Q)^2}{t_{j+1}^\frac{1}{2}},$$
where $2Q_0$ is the ball whose center is the origin and the radius is the same as that of the cube $2Q$. Let $(z,t)(z_0, t_0)^{-1}=(u,\tilde{t})$. We only consider the $\tilde{t}=0$ case. For general $\tilde{t} $ one can follow the proof of \cite[Lemma 1, part iv]{L}. Let $L$ be the left invariant vector field corresponding to the curve $\gamma(\alpha)=\alpha \frac{(u,0)}{|u|}$, $\alpha\in \R$. Then from the fundamental theorem of calculus, we have
\begin{align*}
&\Big(\int_{H^n\setminus 2Q_0} \rho(w,s)^\frac{n+\frac{3}{2}}{2}|k_{j,1}\big((u,0)(w,s)\big)-k_{j,1}(w,s)|^2dwds\Big)^\frac{1}{2}\\
&\lesssim \Big(\int_{H^n\setminus 2Q_0} \rho(w,s)^\frac{n+\frac{3}{2}}{2}\Big|\int_0^{|u|}Lk_{j,1}\big(\gamma(\alpha)(w,s)\Big)d\alpha\Big|^2dwds\big)^\frac{1}{2}\\
&\lesssim\int_0^{|u|}\Big(\int_{H^n\setminus 2Q_0} \rho(w,s)^\frac{n+\frac{3}{2}}{2}|Lk_{j,1}\big(\gamma(\alpha)(w,s)\big)|^2dwds\Big)^\frac{1}{2}d\alpha.
\end{align*}
As $\rho(w,s)$ and $\rho\left((u,s)(w,s)\right)$ are comparable, $\rho(w,s)$ and $\rho\left(\gamma(\alpha)(w,s)\right)$ are also comparable. Hence the above can be dominated by
$$
\int_0^{|u|}\Big(\int_{H^n\setminus 2Q_0} \rho\big(\gamma(\alpha)(w,s)\big)^\frac{n+\frac{3}{2}}{2}|Lk_{j,1}\big(\gamma(\alpha)(w,s)\big)|^2dwds\Big)^\frac{1}{2}d\alpha.
$$
From the definition of $k_{j,1}$, we have $Lk_{j,1}(w,s)=Lk_j(w,s) \phi(w,s)+ k_j(w,s)L\phi(w,s)$. Now as $|L\phi(w,s)|$ is bounded, the corresponding integral associated to that term can be dominated by
$$
\int_0^{|u|}\Big(\int_{H^n\setminus 2Q_0} \rho\big(\gamma(\alpha)(w,s)\big)^\frac{n+\frac{3}{2}}{2}|k_{j}\big(\gamma(\alpha)(w,s)\big)|^2dwds\Big)^\frac{1}{2}d\alpha.
$$
Using Lemma \ref{lem:main2} this can be further bounded by
$$ t_{j+1}^{-\frac{1}{4}} |u|\lesssim a\, l(Q) t_{j+1}^{-\frac{1}{4}}\lesssim l(Q)^\frac{1}{2} \frac{l(Q)^\frac{1}{2}}{t_{j+1}^\frac{1}{4}}.$$
Using the bounedness of $\phi(w,s)$, we can dominate the integral associated to the term $Lk_j(w,s) \phi(w,s)$ by
$$
C~\int_0^{|u|}\Big(\int_{H^n\setminus 2Q_0} \rho\big(\gamma(\alpha)(w,s)\big)^\frac{n+\frac{3}{2}}{2}|L k_{j,1}\big(\gamma(\alpha)(w,s)\big)|^2dwds\Big)^\frac{1}{2}d\alpha.
$$
As $k_{j,1}$ is supported in the homogenous ball of radius $1$ and $\rho(w,s)$ and $\rho\left(\gamma(\alpha)(w,s)\right)$ are  comparable, the above is less than or equal to
$$
\int_0^{|u|}\Big(\int_{H^n\setminus 2Q_0} \rho\big(\gamma(\alpha)(w,s)\big)^\frac{n+\frac{5}{2}}{2}|Lk_{j}\big(\gamma(\alpha)(w,s)\big)|^2dwds\Big)^\frac{1}{2}d\alpha.
$$
By Lemma \ref{lem:2} the above integral is again bounded by
$$C~\frac{|u|}{t_{j+1}^\frac{1}{4}}\leq C~l(Q)^{\frac{1}{2}}\frac{l(Q)^\frac{1}{2}}{t_{j+1}^\frac{1}{4}}. $$
Hence (\ref{estim:min}) is proved and this completes the proof of (\ref{estim:f1}).

We will now prove that \be\label{estim:f2}|T^{N}_2 f_2(z,t)|\leq C~ \Lambda_2 f(z',t'),\ee
for any $(z,t), (z',t') \in Q$.

Let $s_{j}=t_j^\frac{1}{2}$. As for all $(w,s)\in H^n\setminus 3Q$, $\rho\left((z,t)(w,s)^{-1}\right))$ and $\rho\left((z',t')(w,s)^{-1}\right))$ are comparable, by Cauchy-Schwarz inequality we can dominate $|T^{N}_2 f(z,t)|$ by
\begin{multline*}
\sum_{j=1}^N\Big(\int_{H^n\setminus 3Q}\big(1+ s_{j+1}^{-2}\rho\big((z,t)(w,s)^{-1}\big)\big)^{\frac{n+\frac{3}{2}}{2}}|k_{j,2}\big((z,t)(w,s)^{-1}\big)|^2dwds\Big)^\frac{1}{2}\\
\times\bigg(\int_{H^n\setminus 3Q}\frac{|f(w,s)|^2}{\big(1+ s_{j+1}^{-2}\rho\big((z',t')(w,s)^{-1}\big)\big)^{\frac{n+\frac{3}{2}}{2}}}dwds\bigg)^\frac{1}{2}.
\end{multline*}
Since $k_{j,2}$ is supported outside homogeneous ball of radius 1 and centre at the origin and $s_{j+ 1}>1$, using Lemma \ref{lem:main2}, we have
\begin{align*} &\int_{H^n\setminus 3Q} (1+s_{j+1}^{-2}\rho\big((z,t)(w,s)^{-1}\big))^{\frac{n+\frac{3}{2}}{2}}|k_{j,2}\big((z,t)(w,s)^{-1}\big)|^2dwds\\
&\leq s_{j+1}^{-n-\frac{3}{2}}\int_{H^n} \rho\big((z,t)(w,s)^{-1}\big)^{\frac{n+\frac{3}{2}}{2}}|k_j\big((z,t)(w,s)^{-1}\big)|^2dwds\\
&\leq C~  s_{j+1}^{-(n+\frac{3}{2})}t_{j+1}^{\frac{1}{2}}\\
&\leq C~s_{j+1}^{-(n+\frac{1}{2})}.\end{align*}
it is well-known that $\bigg(\int_{H^n\setminus 3Q}\frac{|f(w,s)|^2}{\big(1+ s_{j+1}^{-2}\rho\big((z',t')(w,s)^{-1}\big)\big)^{\frac{n+\frac{3}{2}}{2}}}dwds\bigg)^\frac{1}{2}$ is dominated by
$s_{j+1}^{\frac{n+2}{2}}\Lambda_2f(z',t').$
Hence $|T^{N}_2 f_2(z,t)|$ is bounded by
$$C\sum_{j=1}^N s_{j+1}^{\frac{3}{4}}\Lambda_2f(z,t)\leq C~\Lambda_2f(z',t') $$
where the constant is independent of $N$. Therefore, we have proved (\ref{estim:f2}).

Now, let $(z,t)\in Q$. Then,
\begin{align*}
|\mathcal{T}^N_Q f(z_0,t_0)|&\leq |T^N f_2(z_0,t_0)|\\
&\leq |T^{N} f_2(z_0,t_0)-T^{N} f_2(z,t)|+|T^{N} f_2(z,t)|\\
&\leq |T^{N}_1 f_2(z_0,t_0)-T^{N}_1 f_2(z,t)|+|T^{N}_2 f_2(z_0,t_0)|\\&\quad +|T^{N}_2 f_2(z,t)|+|T^{N} f(z,t)|+|T^{N} f_1(z,t)|\\
&\leq C~ \big(\Lambda_2 f(z_0, t_0)+|T^N f(z,t)|+|T^{N} f_1(z,t)|\big).
\end{align*}
Taking the integral over the region $(z,t)\in Q$ on both sides and dividing by $\frac{1}{|Q|}$ we have
$$|\mathcal{T}^N_{Q}f(z_0,t_0)|\leq C~ \Big(\Lambda_2 f(z_0, t_0)+\Lambda(T^N f)(z_0,t_0)|+\frac{1}{|Q|}\int_{Q}|T^{N} f_1(z,t)|\Big).$$
We can estimate $\frac{1}{|Q|}\int_{Q}|T^{N} f_1(z,t)|$ by using the $L^2$ boundedness of $T^{N}$ as follows
\begin{align*}\frac{1}{|Q|}\int_{Q}|T^{N} f_1(z,t)|dzdt&\leq \Big(\frac{1}{|Q|}\int_{Q}|T^{N} f_1(z,t)|^2dzdt\Big)^\frac{1}{2}\\
&\leq \Big(\frac{1}{|Q|}\int_{H^n}|T^{N} f_1(z,t)|^2dzdt\Big)^\frac{1}{2}\\
&\leq \Big(\frac{1}{|Q|}\int_{3Q}|f(z,t)|^2dzdt\Big)^\frac{1}{2}\\
&\leq \Lambda_2 f (z_0, t_0).
\end{align*}
So, finally we have
$$|\mathcal{T}^N_{Q}f(z_0,t_0)|\leq C~ \big(\Lambda_2 f(z_0, t_0)+\Lambda(T^N f)(z_0,t_0)\big).$$
Taking supremum over all cubes containing $(z_0, t_0)$ we get
$$\mathcal{T}^{N*}f(z_0,t_0)\leq C~ \big(\Lambda_2 f(z_0, t_0)+\Lambda(T^N f)(z_0,t_0)\big).$$
Hence (\ref{estim:T*}) is proved. Inequality (\ref{estim:T*2}) will follow from (\ref{estim:T*}) using the weak (2,2) bounedness of $\Lambda_2$ and $T^N$.
\end{proof}

Let us consider the following maximal function
$$\mathcal{M}_{T^N}f(z,t)= \sup_{(z,t)\in Q} \esssup_{(\xi,s)\in Q} |T^N(f\chi_{H^n\setminus 3Q})(\xi,s)|. $$
Here supremum is taken over all the cubes which contains $(z,t)$. Also for any cube $Q_0$ define
$$\mathcal{M}_{T_{Q_0}^N}f(z,t)= \sup_{(z,t)\in Q, Q\subset Q_0} \esssup_{(\xi,s)\in Q} |T^N(f\chi_{H^n\setminus 3Q})(\xi,s)|$$
The next lemma will be used to prove \ref{th:Main1}.
\begin{lem}\label{lem:pointwise}Let $M$ satisfies the hypothesis of the Theorem \ref{th:Main1}. Then
\begin{enumerate}\item For a.e $(z,t)\in Q$,
$$|T^N(f\chi_{3Q_0})(z,t)|\leq C~ \big(f(z,t)+\mathcal{M}_{T_{Q_0}^N}f(z,t)\big).$$
\item $\mathcal{M}_{T^N}f(z,t)\leq  C~\big(\Lambda_2f(z,t)+\mathcal{T}^{N*}f(z,t)\big)$
for any $(z,t)\in H^n$.
\end{enumerate}
Here $C$ depends only on $n$ and $T$, and not on $N$.

\end{lem}
\begin{proof} The proof of the first part is same as the proof of part (i) of Lemma 3.2 of \cite[Lemma 3.2, part(i)]{AL1}. So, we will prove part (2). Let us consider a dyadic cube $Q$ containing $(z,t)\in H^n$. Let $(\xi,s)\in Q$. Then
$$|T^{N}f\chi_{H^n\setminus3Q}(\xi, s)|\leq |T^{N}f\chi_{H^n\setminus 3Q}(\xi, s)-T^{N}f\chi_{H^n\setminus 3Q}(z, t)|
+|T^{N}f\chi_{H^n\setminus 3Q}(z,t)|$$
Now,
$|T^{N}f\chi_{H^n\setminus 3Q}(\xi, s)-T^{N}f\chi_{H^n\setminus 3Q}(z, t)|$ can be dominated by $$|T_1^{N}f\chi_{H^n\setminus 3Q}(\xi, s)-T_1^{N}f\chi_{H^n\setminus 3Q}(z, t)|+|T^N_2f\chi_{H^n\setminus 3Q}(\xi, s)|+|T_2^{N}f\chi_{H^n\setminus 3Q}(z, t)|.$$
We have seen in the proof of Theorem \ref{thm:Strong} that all the above terms can be dominated by $\Lambda_2 f(z,t).$
On the other hand $|T^{N}f\chi_{H^n\setminus 3Q}(z,t)|$ is bounded by $\mathcal{T}^{N*}f(z,t)$.
Hence, we have
$$\mathcal{M}_{T^N}f(z,t)\leq C~ \big(\Lambda_2f(z,t)+ \mathcal{T}^{N*}f(z,t)\big).$$
\end{proof}
\begin{proof}[Proof of Theorem \ref{th:Main1}] Since we prove that $ \mathcal{T}^{N*}$ is of weak type $(2,2)$, by our previous lemma $\mathcal{M}_{T^N}$ is of weak type $(2,2)$.
 Let $Q$ be a cube in $H^n$. We shall find a sparse family $\mathcal{F}\subset \mathcal{D}(Q_0)$ such that
\be\label{estim:pointwise}|T(f\chi_{3Q_0})(z,t)|\leq C~ \sum_{Q\subset \mathcal{F}} \Big(\frac{1}{|3Q|}\int_{3Q}|f|^2\Big)^\frac{1}{2}\chi_Q(z,t)\ee
for a.e $(z,t)\in Q_0$. The above can be proved using exactly the similar idea used in proving \cite[Theorem 3.1]{AL1}.
To prove (\ref{estim:pointwise}), we will use a recursive method. We will first find a pairwise disjoint cubes $ Q_j^1\in \mathcal{D}(Q_0)$ such that $\sum |Q^1_j|\leq \frac{1}{2}|Q_0|$ and
\be\label{estim:fs}|T(f\chi_{3Q_0})(z,t)|\chi_{Q_0}(z,t)\leq C \Big(\frac{1}{|3Q_0|}\int_{3Q_0}|f|^2\Big)^\frac{1}{2}\chi_{Q_0}+\sum_j|T(f\chi_{3Q^1_j})|\chi_{Q^1_j}\ee
a.e $(z,t)\in H^n$.
Once we prove this, we will apply the same process on each $Q^1_j$ and continue the recursive process which finally leads to (\ref{estim:pointwise}).

As $\mathcal{M}_{T^N}$ is of weak type $(2,2)$, we can choose $\alpha\in R$ (depending only on $n$) such that the measure of the set
\begin{multline*}
E=\Big\{(z,t)\in Q_0 : |f(z,t)|> \alpha \Big(\frac{1}{|3Q_0|}\int_{3Q_0} |f|^2\Big)^\frac{1}{2}\Big\}\\
\cup\Big\{(z,t)\in Q_0 : \mathcal{M}_{T^N_{Q_0}}f(z,t)> \alpha \Big(\frac{1}{|3Q_0|}\int_{3Q_0} |f|^2\Big)^\frac{1}{2}\Big\}
\end{multline*}
is less than $\frac{1}{4\epsilon}|Q_0|$, where the constant $\epsilon$ is coming from the definition of the dyadic cubes and $\eta$.

Applying Calderon--Zygmund decomposition to the function $\chi_E$ on $Q_0$ at height $\frac{\epsilon}{2}$ we get pairwise disjoint cubes $Q^1_j\in \mathcal{D}(Q_0)$ such that
\begin{enumerate}
\item $\frac{\epsilon}{2}|Q^1_j|\leq |Q^1_j\cap E|\leq \frac{1}{2}|Q^1_j|$
\item $|E\setminus \cup_j Q^1_j|=0.$
\end{enumerate}
From (1), we have $Q^1_j\cap E^c\neq \emptyset$. Using the fact $|E|\leq \frac{\epsilon}{4}|Q_0|$, one can observe that
$\sum_j|Q^1_j|\leq \frac{1}{2}|Q_0|$. Therefore
$$\esssup_{(\xi,s)\in Q^1_j}|T(f\chi_{3Q_0\setminus 3Q^1_j}(\xi, s))|\leq C~ \Big(\frac{1}{|3Q_0|}\int_{3Q_0}|f|^2\Big)^\frac{1}{2}.$$
Also, using Lemma \ref{lem:pointwise} we can easily see that if $x\in Q_0\setminus \cup_j Q^1_j$, then
$$|T(f\chi_{3Q_0})(x)|\leq C~\Big(\frac{1}{|3Q_0|}\int_{3Q_0}|f|^2\Big)^\frac{1}{2}.$$

Using above result we get
\begin{align*}
|T(f\chi_{3Q_0})|\chi_{Q_0}&\leq |T(f\chi_{3Q_0})|\chi_{Q_0\setminus Q^1_j}+\sum_j |T(f\chi_{3Q_0})|\chi_{ Q^1_j}\\
&\leq |T(f\chi_{3Q_0})|\chi_{Q_0\setminus Q^1_j}+\sum_j |T(f\chi_{3Q_0\setminus 3Q^1_j})|\chi_{ Q^1_j}+\sum_j |T(f\chi_{ 3Q^1_j})|\chi_{ Q^1_j}\\
&\leq \Big(\frac{1}{|3Q_0|}\int_{3Q_0}|f|^2\Big)^\frac{1}{2}+\sum_j |T(f\chi_{ 3Q^1_j})|\chi_{ Q^1_j}.
\end{align*}
Hence, (\ref{estim:fs}) is proved.

Now, let us assume that $f$ is compactly supported. Let us consider any ball $B$ which contains the support of $f$. By \cite[Theorem 4.1]{H2}, there exists finite number of dyadic decompositions $\{S^l:l=1,2,\cdots, L\}$ and a cube $Q\in S^l$ for some $l=1,2,\cdots, L$ such that $B\subset Q$.
As $Q$ contains the support of $f$ from (\ref{estim:pointwise}) we have
\begin{multline*}
\|T^Nf(z,t)\|_{L^p(B, w)}\leq \|T^Nf(z,t)\|_{L^p(Q, w)}\\
\leq C~\|\sum_{Q\subset \mathcal{F}} \Big(\frac{1}{|3Q|}\int_{3Q}|f|^2\Big)^\frac{1}{2}\|_{L^p(Q, w)}\leq C~\|f\|_{L^p(w)}.
\end{multline*}
Since, the above is true for any cube $B$ which contains the support of $f$, we have
$$\|T^Nf(z,t)\|_{L^p(w)}\leq C~\|f\|_{L^p(w)}$$
for all $f\in C_c^\infty(H^n)$ and hence for all $f\in L^p(w)$

Of course, the constant appearing here is independent of $N$. Now, as $T^N$ converges to $T_M$ in $L^2$, there exists a subsequence which converges to $T_M$ almost everywhere. Thus, we can conclude that
$$\|T_Mf(z,t)\|_{L^p(w)}\leq C~\|f\|_{L^p(w)}$$
for $w\in A_{\frac{p}{2}}$, $2<p<\infty$.
 Hence Theorem \ref{th:Main1} is proved.
\end{proof}

\subsection*{Acknowledgements}
 This work is supported by Inspire fellowship from Department of Science and Technology. The author is extremely thankful to Dr. Rahul Garg, Dr. Luz Roncal, Prof. Sundaram Thangavelu and the referee for their careful reading of the manuscript, for giving him a lot of useful suggestions and pointing out typos. The author would also like to thank to Prof. The Anh Bui for pointing out an oversight.

\end{document}